%% file: HFEM.tex
\documentclass[11pt]{amsart}

\usepackage[utf8]{inputenc}
\usepackage[T1]{fontenc}

\usepackage{amsmath,amssymb,amsfonts,textcomp,amsthm,xifthen,graphicx,color,pgfplots}

\usepackage{enumerate}
\usepackage{enumitem}

\usepackage{pgfplots}
\usepgfplotslibrary{groupplots}
\usepgfplotslibrary{colorbrewer}
\pgfplotsset{compat=newest}
\usepgfplotslibrary{external}
\usepackage{fullpage}

\usepackage[utf8]{inputenc}

\usepackage[pdftex,
  pdfauthor={Thomas F\"uhrer and Diego Paredes},
            pdftitle={Hybrid methods for reaction-dominated diffusion problem},
            ]{hyperref}

\input{header}

\begin{document}

\title{Robust hybrid finite element methods for reaction-dominated diffusion problems}
\date{\today}

\author{Thomas F\"uhrer}
\address{Facultad de Matem\'{a}ticas, Pontificia Universidad Cat\'{o}lica de Chile, Santiago, Chile}
\email{tofuhrer@mat.uc.cl}
\author{Diego Paredes}
\address{Departamento de Ingenier\'ia Matem\'atica  and CI$^2$MA,  Universidad de Concepci\'on, Concepci\'on, Chile}
\email{dparedes@udec.cl}

\thanks{{\bf Acknowledgment.} 
This work was supported by ANID through FONDECYT project 1210391 (TF) and FONDECYT project 1181572 (DP).  Parts of the presented research have been conducted while DP was visiting the mathematics faculty of Pontificia Universidad Cat\'{o}lica de Chile during December 2023.}

\begin{abstract}
  For a reaction-dominated diffusion problem we study a primal and a dual hybrid finite element method where weak continuity conditions are enforced by Lagrange multipliers. 
  Uniform robustness of the discrete methods is achieved by enriching the local discretization spaces with modified face bubble functions which decay exponentially in the interior of an element depending on the ratio of the singular perturbation parameter and the local mesh-size.
  A posteriori error estimators are derived using Fortin operators. They are robust with respect to the singular perturbation parameter.
  Numerical experiments are presented that show that oscillations, if present, are significantly smaller then those observed in common finite element methods.
\end{abstract}
\maketitle

\section{Introduction}
\noindent
The aim of this article is to study robust hybrid finite element methods for the reaction-diffusion problem
\begin{subequations}\label{eq:model}
\begin{alignat}{2}
  -\varepsilon^2 \Delta u + u &= f &\quad&\text{in }\Omega,\label{eq:model:a} \\
  u &= 0 &\quad&\text{on } \Gamma :=\partial\Omega,
\end{alignat}
\end{subequations}
where $f\in L_2(\Omega)$ and $\Omega\subseteq \R^d$ ($d=2,3$) denotes an open Lipschitz domain with polygonal/polyhedral boundary $\Gamma = \partial\Omega$ and $0<\varepsilon\leq 1$.
We note that~\eqref{eq:model} admits a unique weak solution $u\in H_0^1(\Omega)$.
Particularly, the reaction-dominated regime where $0<\varepsilon\ll 1$ is of interest and poses several challenges for numerical methods. 
Indeed, discretizing the weak formulation with piecewise polynomials that are continuous on $\Omega$ gives approximations with non-physical oscillation phenomena in the region of the boundary (or interior) layer. 
For a detailed description and numerical treatment of singularly perturbed problems we refer the interested reader to~\cite{RoosStynesTobiska96}.
Different numerical methods for problem~\eqref{eq:model} have been studied extensively, see, e.g.,~\cite{LinStynes12,DPGrefusion,CaiKu20,MR2452853,MR2583792,MR2557053,MR2454526,MR3358551,MR2164094}, to name a few. 

Finite element analysis based on layer-adapted meshes is popular, cf.~\cite{RoosStynesTobiska96}, but needs a priori knowledge of the location of layers. 
An alternative are adaptive methods driven by robust a posteriori error estimators, see, e.g.~\cite{VerfReactDiff}.
Other finite element methods are based on the use of balanced norms see, e.g.,~\cite{LinStynes12,DPGrefusion,CaiKu20}. In~\cite{LinStynes12} the authors propose a finite element method that is inspired by least-squares finite element methods and uses the primal variable $u$ and an additional flux variable.
The authors of~\cite{DPGrefusion} propose and analyze a discontinuous Petrov--Galerkin method (DPG) with optimal test functions for an ultra-weak formulation of~\eqref{eq:model} and show that their method is robust with respect to a balanced norm.

On the other hand, \cite{MR0701094} introduced the idea to modify the approximation spaces by using the differential operator of the problem to define improved basis functions.  This concept marked a turning point in the field and initiated a highly productive avenue of research for instance  \cite{MR1291139},~\cite{MR2203943} and~\cite{MR2863007}.

Furthermore, in~\cite{CaiKu20} a dual formulation of~\eqref{eq:model} is derived that allows to eliminate the primal variable from the system.   
In this work (see \eqref{eq:dualdglm}) we introduce a  hybrid FEM based on a similar idea but with a smaller number of degrees of freedom.

Hybrid finite element methods (HFEM) are popular as they allow to relax constraints on the discrete trial/test functions, e.g., alleviate interelement continuities, see, e.g.,~\cite{RT77}.
This feature is particularly useful for discretizations of fourth-order problems~\cite{Pian72}.
A type of HFEM for singularly perturbed problems has been studied in~\cite{MR2640986} and \cite{MR3336297} presents a multiscale finite element method based on the hybrid formulation of \eqref{eq:model} with an approximation space inspired by the basis functions defined in~\cite{MR2863007}.  
We note that our proposal differs from those in \cite{MR2640986} and \cite{MR3336297} in the definition of the approximation spaces. Specifically, we have designed the basis functions to demonstrate the robustness of the resulting numerical method with respect to small values of $\varepsilon$.

\subsection{Novel contributions}
In the work at hand we study a primal and a dual HFEM.
The first one is derived by multiplying~\eqref{eq:model:a} by some discontinuous test function and by introducing a Lagrange multiplier $\lambda$ (formally, $\lambda$ is the normal trace of $\varepsilon \nabla u$ on each element boundary). The resulting equivalent formulation is of mixed form: Find $(u,\lambda)\in U\times \Lambda$ such that
\begin{alignat*}{2}
  &a(u,v) + b(\lambda,v) &\,=\,& \ip{f}{v}_\Omega, 
  \\
  &b(\mu,u) &\,=\,&0
\end{alignat*}
for all $(v,\mu)\in U\times \Lambda$. Here, $a(u,v) = \varepsilon^2\ip{\pwnabla u}{\pwnabla v}_\Omega + \ip{u}v_\Omega$ and $b(\lambda,v) = -\varepsilon\dual{\lambda}{v}_{\partial\cT}$, and $\cT$ denotes a shape-regular mesh of simplices.
For a precise definition of the spaces and bilinear forms we refer to Section~\ref{sec:dglm}.
We replace $U\times \Lambda$ by some finite element spaces $U_h\times \Lambda_h$. 
By the theory of mixed finite element methods~\cite{BoffiBrezziFortin} the pair $(U_h,\Lambda_h)$ needs to satisfy the discrete $\inf$--$\sup$ condition
\begin{align}\label{eq:discreteInfSup}
  \inf_{0\neq \lambda_h\in \Lambda_h}\sup_{0\neq v_h\in U_h} \frac{b(\lambda_h,v_h)}{\|\lambda_h\|_{\Lambda,\varepsilon}\|v_h\|_{U,\varepsilon}} \geq \gamma_h >0
\end{align}
to ensure solvability of the discretized mixed scheme. 

For quasi-optimality results (and others) with constants independent of the mesh-size and the singular perturbation parameter, one also requires that $\gamma_h$ does not deteriorate when $h\to 0$ or $\varepsilon\to 0$.
Choosing $\Lambda_h$ to be a piecewise polynomial space the last requirement (i.e., $\gamma_h$ independent of $\varepsilon$) is not satisfied if $U_h = P^p(\cT)$  (the space of piecewise polynomials of degree $\leq p$ on the shape-regular mesh of simplices $\cT$) and $\varepsilon\ll h$. 
To be more precise, following the analysis from~\cite{FuehrerHeuerFortin} one can prove (details not shown here) that $\gamma_h\eqsim \min\{1,\sqrt{\varepsilon/h}\}$. 
This means that on coarse meshes ($h\eqsim 1$) and small values of the perturbation parameter ($\varepsilon\ll 1$), the $\inf$--$\sup$ constant is small, yielding possible non-trustworthy approximations.
To avoid dependence on the ratio $\varepsilon/h$ we follow the recent work~\cite{FuehrerHeuerFortin} and choose $U_h$ to be a polynomial space enriched with modified face bubble functions. These are defined to be equal to polynomial face bubble functions on the boundary of elements but decay exponentially in the interior of elements depending on the (local) ratio $\varepsilon/h$, see Figure~\ref{fig:bubbles} for a visualization. 
Their use allows to prove existence of a Fortin operator which is bounded independently of $\varepsilon$ and $h$ and consequently $\gamma_h \eqsim 1$. 

For our second method we consider the flux given by $\varepsilon\nabla u = \bsigma$ as independent variable. This allows to eliminate the primal variable from the system. Details on this method and its derivation are found below (Section~\ref{sec:dualdglm}). 
We stress that~\cite{CaiKu20} propose and thoroughly analyze a continuous Galerkin method based on the flux variable (with a different defined flux $\bsigma = \nabla u$).

For both proposed methods we derive and study a posteriori error estimators that are robust with respect to $\varepsilon$ and can be used to steer adaptive mesh-refinements. 
The error indicators are defined by using Fortin operators with additional orthogonality properties. 
We show that the estimators are reliable (for the first estimator under the assumption that $\Omega$ is convex) and locally efficient up to best approximation terms and data oscillations. 
For one of the efficiency estimates we employ the usual bubble function technique~\cite{Verf94} but replace the polynomial face bubble functions by the modified ones, see Lemma~\ref{lem:effTraceSigma}.

By taking local Schur complements our methods reduce to efficient numerical methods as the dimension of the system is equal to the dimension of the discrete Lagrange multiplier space. 

\subsection{Overview}
The remainder of this work is organized as follows: 
In Section~\ref{sec:analysis} we introduce some notation, spaces and the HFEMs together with an analysis thereof. 
Section~\ref{sec:stability} discusses discretization and conditions on the discrete spaces to ensure stability. 
In particular, we introduce the modified face bubble functions mentioned above. 
An a posteriori error estimator for each of the methods is derived in Section~\ref{sec:apost}. 
This work is concluded by numerical experiments presented in Section~\ref{sec:numeric}.

Throughout this article by $A\lesssim B$ we abbreviate $A \leq C\cdot B$ for $A,B>0$ and $C>0$ a constant independent of $\varepsilon$ and the mesh (except possible its shape-regularity) and other quantities of interest. 
In the same spirit we write $A\eqsim B$ if $A\lesssim B$ and $B\lesssim A$ hold simultaneously.

\section{Derivation and abstract analysis of the hybrid methods}\label{sec:analysis}
In this section we derive and analyze our two numerical schemes, see Section~\ref{sec:mixedLM} and~\ref{sec:dglm} for the primal hybrid FEM (PHFEM) and Section~\ref{sec:dualmixedLM} and~\ref{sec:dualdglm} for the dual hybrid FEM (DHFEM).
Before that we introduce some notation, spaces and auxiliary results in Section~\ref{sec:spaces}.

\subsection{Sobolev spaces and traces}\label{sec:spaces}
For $T\subset \R^d$ a Lipschitz domain let
\begin{align*}
  H^1(T) &= \set{v\in L_2(T)}{\nabla v\in L_2(T)^d}, \\
  \Hdivset{T} &= \set{\btau\in L_2(T)^d}{\div\btau \in L_2(T)},
\end{align*}
with trace spaces $H^{1/2}(\partial T)$ and $H^{-1/2}(\partial T)$, respectively. 
The duality between $H^{-1/2}(\partial T)$ and $H^{1/2}(\partial T)$ is denoted by $\dual{\mu}{v}_{\partial T}$ and extends the $L_2(\Gamma)$ scalar product. 
The canonic inner product on $L_2(T)$ is denoted by $\ip{\cdot}\cdot_T$ with induced norm $\|\cdot\|_T$.

The normal trace operator $\trdiv_T\colon \Hdivset{T}\to \big(H^1(T)\big)'$ is given by
\begin{align*}
  \dual{\trdiv_T\btau}{v}_{\partial T} = \ip{\div\btau}{v}_T + \ip{\btau}{\nabla v}_T \quad\forall \btau\in \Hdivset{T}, v\in H^1(T).
\end{align*}
Note that for sufficiently smooth $\btau$ we can identify this operator with $\btau\cdot\normal_T$ with $\normal_T$ being the (exterior) normal on $\partial T$. Thus, we may use $\btau\cdot\normal_T$ instead of $\trdiv_T\btau$ from this point on.

We equip spaces $H^1(T)$ and $\Hdivset{T}$ with the parameter dependent norms ($0<\varepsilon\leq 1$), 
\begin{align*}
  \|v\|_{1,\varepsilon,T}^2 = \|v\|_T^2 + \varepsilon^2 \|\nabla v\|_T^2, \qquad
  \|\btau\|_{\div,\varepsilon,T}^2 = \|\btau\|_T^2 + \varepsilon^2 \|\div\btau\|_T^2.
\end{align*}

Let $\cT$ denote a finite covering of $\Omega$ into Lipschitz domains, i.e., $\overline\Omega = \bigcup_{T\in\cT} \overline{T}$.
The broken variants of the spaces above are denoted by $U:=H^1(\cT)$, $\Sigma:=\Hdivset{\cT}$ and the global trace operators $\trdiv_{\cT}\colon \Hdivset\Omega\to \big(H^1(\cT)\big)'$ resp. $\trnabla_{\cT}\colon H^1(\Omega)\to \big(\Hdivset\cT\big)'$ are given by
\begin{align*}
  \dual{\trdiv_{\cT}\btau}{v}_{\partial \cT} &= \sum_{T\in\cT} \dual{\trdiv_T(\btau|_T)}{v}_{\partial T}
  \quad\forall\btau\in\Hdivset{\Omega},v\in H^1(\cT), \quad\text{resp.}\\
  \dual{\trnabla_{\cT}v}{\btau}_{\partial \cT} &= \sum_{T\in\cT} \dual{\trnabla_Tv}{\btau}_{\partial T} := \sum_{T\in\cT}\dual{\btau\cdot\normal_T}{v}_{\partial T}
  \quad\forall v\in H^1(\Omega), \, \btau\in \Hdivset\cT.
\end{align*}
Furthermore, let $\Lambda := H^{-1/2}(\cT) = \trdiv_{\cT}(\Hdivset\Omega)$ and $W := H^{1/2}_{00}(\cT) = \trnabla_{\cT}(H_0^1(\Omega))$. 
We equip the latter two spaces with their respective minimum energy norms
\begin{align*}
  \|\lambda\|_{\Lambda,\varepsilon} &= \inf\set{\|\btau\|_{\div,\varepsilon,\Omega}}{\trdiv_{\cT}\btau = \lambda}, \\
  \|w\|_{W,\varepsilon} &= \inf\set{\|v\|_{1,\varepsilon,\Omega}}{\trnabla_{\cT}v = w}.
\end{align*}
which makes $\Lambda$ resp. $W$ a Banach space.
Norms on the broken spaces are denoted by
\begin{align*}
  \|v\|_{U,\varepsilon}^2 := \|v\|_{1,\varepsilon,\cT}^2 := \sum_{T\in\cT} \|v\|_{1,\varepsilon,T}^2,
  \qquad
  \|\btau\|_{\Sigma,\varepsilon}^2 := \|\btau\|_{\div,\varepsilon,\cT}^2 := \sum_{T\in\cT} \|\btau\|_{\div,\varepsilon,T}^2
\end{align*}

The trace operators $\trdiv_{\cT}$ and $\trnabla_\cT$ are bounded, i.e., 
\begin{align*}
  |\varepsilon\dual{\trdiv_{\cT}\btau}{v}_{\partial\cT}| = |\ip{\varepsilon\div\btau}{v}_\Omega+\ip{\btau}{\varepsilon\pwnabla v}|
  \leq \|\btau\|_{\div,\varepsilon,\Omega}\|v\|_{1,\varepsilon,\cT} = \|\btau\|_{\div,\varepsilon,\Omega}\|v\|_{U,\varepsilon}
\end{align*}
for all $\btau\in \Hdivset\Omega$, $v\in U=H^1(\cT)$ and 
\begin{align*}
|\varepsilon\dual{\trnabla_\cT v}{\btau}_{\partial\cT}| = |\ip{\varepsilon\pwdiv\btau}v_\Omega + |\ip{\btau}{\varepsilon\nabla v}_\Omega|
\leq \|v\|_{1,\varepsilon,\Omega}\|\btau\|_{\Sigma,\varepsilon}
\end{align*}
for all $v\in H_0^1(\Omega)$, $\btau\in \Sigma=\Hdivset\cT$.
We define $\pwnabla \colon H^1(\cT)\to L_2(\Omega)^d$ by $(\pwnabla v)|_T = \nabla(v|_T)$ 
and $\pwdiv\colon \Hdivset\cT\to L_2(\Omega)$ by $(\pwdiv\btau)|_T = \div(\btau|_T)$ for all $T\in\cT$.

The next result relates norms $\|\cdot\|_{\Lambda,\varepsilon}$ and $\|\cdot\|_{W,\varepsilon}$ to  dual norms. 
The proof follows along the lines of techniques developed for the analysis of discontinuous Petrov--Galerkin methods (DPG), see, e.g.~\cite[Theorem~2.3]{BrokenSpaces}.
While~\cite{BrokenSpaces} considers non-singular perturbed problems, i.e., $\varepsilon=1$, their proof can be easily adapted to the present situation ($\varepsilon\neq 1$). We thus omit further details.
\begin{lemma}\label{lem:traceinfsup}
  For all $\lambda\in \Lambda$ and all $w\in W$ we have that
  \begin{align*}
    \|\lambda\|_{\Lambda,\varepsilon} &= \sup_{0\neq v\in U} \frac{\varepsilon\dual{\lambda}{v}_{\partial\cT}}{\|v\|_{U,\varepsilon}}, \\
    \|w\|_{W,\varepsilon} &= \sup_{0\neq \btau\in \Sigma} \frac{\varepsilon\dual{w}{\btau}_{\partial\cT}}{\|\btau\|_{\Sigma,\varepsilon}}.
  \end{align*}
\end{lemma}

The next result precisely characterizes continuity of elements in $U$ and $\Sigma$. It can be found in~\cite[Theorem~2.3]{BrokenSpaces}.
\begin{lemma}\label{lem:cont}
  Let $v\in U$ and $\btau\in \Sigma$ be given. The following equivalences hold:
  \begin{align*}
    v\in H_0^1(\Omega) &\quad\Longleftrightarrow\quad \dual{\mu}{v}_{\partial\cT} = 0 \quad\forall \mu\in \Lambda, \\
    \btau\in \Hdivset\Omega &\quad\Longleftrightarrow\quad \dual{w}{\btau\cdot\normal}_{\partial\cT} = 0 \quad\forall w\in W.
  \end{align*}
\end{lemma}

\subsection{Primal hybrid variational formulation}\label{sec:mixedLM}
We test~\eqref{eq:model:a} with $v\in H^1(T)$ and integrate by parts to obtain
\begin{align*}
  \varepsilon^2\ip{\nabla u}{\nabla v}_T -\varepsilon^2\dual{\nabla u\cdot\normal_T}{v}_{\partial T} + \ip{u}v_T = \ip{f}{v}_T. 
\end{align*}
Summing over all $T\in\cT$ and introducing $\lambda = \trdiv_{\cT}\varepsilon\nabla u$ we get the system
\begin{subequations}
\begin{align*}
  \varepsilon^2\ip{\pwnabla u}{\pwnabla v}_\Omega + \ip{u}v_\Omega - \varepsilon\dual{\lambda}{v}_{\partial\cT} = \ip{f}{v}_{\cT}.
\end{align*}
To ensure conformity of solutions we include continuity constraints by imposing
\begin{align*}
  \dual{\mu}{u}_{\partial\cT} = 0 \quad\forall \mu\in \Lambda.
\end{align*}
\end{subequations}

Introducing the bilinear forms 
\begin{align*}
  a(u,v) &= \varepsilon^2\ip{\pwnabla u}{\pwnabla v}_\Omega + \ip{u}v_\Omega, \\
  b(\lambda,v) &= -\varepsilon\dual{\lambda}{v}_{\partial\cT}
\end{align*}
for all $u,v\in U=H^1(\cT)$, $\lambda\in \Lambda=H^{-1/2}(\cT)$ the primal hybrid variational formulation reads: Find $(u,\lambda)\in U\times \Lambda$ such that
\begin{subequations}\label{eq:dglm}
\begin{alignat}{2}
  &a(u,v) + b(\lambda,v) &\,=\,& \ip{f}{v}, \label{eq:dglm:a}
  \\
  &b(\mu,u) &\,=\,&0 \label{eq:dglm:b}
\end{alignat}
\end{subequations}
for all $(v,\mu)\in U\times \Lambda$.

Let us first show that formulation~\eqref{eq:dglm} is equivalent to problem~\eqref{eq:model}.
\begin{proposition}
  Let $f\in L_2(\Omega)$. Problems~\eqref{eq:model} and~\eqref{eq:dglm} are equivalent in the following sense: If $u\in H_0^1(\Omega)$ solves~\eqref{eq:model} then $(u,\trdiv_{\cT}\varepsilon\nabla u)\in U\times \Lambda$ solves~\eqref{eq:dglm}. 
  Conversely, if $(u,\lambda)\in U\times \Lambda$ solves~\eqref{eq:dglm} then $u\in H_0^1(\Omega)$ and $u$ solves~\eqref{eq:model}. Furthermore, $\lambda = \trdiv_{\cT}\varepsilon\nabla u$.
\end{proposition}
\begin{proof}
  Let $u\in H_0^1(\Omega)$ be the solution of~\eqref{eq:model}. 
  By Lemma~\ref{lem:cont} we have that $b(\mu,u) = 0$ for all $\mu\in \Lambda$. 
  Note that $u\in H_0^1(\Omega)$ and $\Delta u \in L_2(\Omega)$ by~\eqref{eq:model:a}. 
  Integration by parts on each $T\in\cT$ and setting $\lambda = \trdiv_{\cT}\varepsilon\nabla u$ then shows that $(u,\lambda)$ solves~\eqref{eq:dglm}. 

  Conversely, let $(u,\lambda)\in U\times \Lambda$ denote a solution of~\eqref{eq:dglm}. By~\eqref{eq:dglm:b} and Lemma~\ref{lem:cont} we get that $u\in H_0^1(\Omega)$. Taking $v\in H_0^1(\Omega)$ we, again by Lemma~\ref{lem:cont}, get that $b(\lambda,v) = 0$. Then,~\eqref{eq:dglm:a} yields
  \begin{align*}
    \varepsilon^2\ip{\nabla u}{\nabla v}_\Omega + \ip{u}v_\Omega = \ip{f}v_\Omega \quad\forall v\in H_0^1(\Omega). 
  \end{align*}
  Thus, $u\in H_0^1(\Omega)$ is the weak solution of~\eqref{eq:model}. It remains to prove that $\lambda = \trdiv_{\cT}(\varepsilon\nabla u)$. This follows by integration by parts in~\eqref{eq:model:a} which finishes the proof. 
\end{proof}

In what follows we analyze problem~\eqref{eq:dglm}. Since it is a mixed system we can apply the classic theory from~\cite{BoffiBrezziFortin}. 
\begin{theorem}\label{thm:dglm}
  Problem~\eqref{eq:dglm} is well posed. In particular, for any $f\in L_2(\Omega)$ there exists a unique solution $(u,\lambda)\in U\times \Lambda$ with
  \begin{align*}
    \|u\|_{U,\varepsilon} + \|\lambda\|_{\Lambda,\varepsilon} \lesssim \|f\|_\Omega. 
  \end{align*}
  The involved constant is independent of $\varepsilon$ and $\cT$. 
\end{theorem}
\begin{proof}
  First, note that $a(\cdot,\cdot)$ and $b(\cdot,\cdot)$ are bounded bilinear forms and that $v\mapsto \ip{f}v_\Omega$ defines a bounded linear functional on $U$. Furthermore, Lemma~\ref{lem:traceinfsup} shows that
  \begin{align*}
    \sup_{0\neq v\in U} \frac{b(\mu,v)}{\|v\|_{U,\varepsilon}} = \|\mu\|_{\Lambda,\varepsilon} \quad\forall \mu \in \Lambda. 
  \end{align*}
  Since $a(\cdot,\cdot)$ is an inner product which induces the norm $\|\cdot\|_{U,\varepsilon}$ the proof is finished. 
\end{proof}

\subsection{PHFEM}\label{sec:dglm}
Let $U_h\subset U$, $\Lambda_h\subset \Lambda$ denote finite-dimensional subspaces. We say that $\Pi_U\colon U\to U_h$ is a Fortin operator if there exists $C_U>0$ with
\begin{align}\label{eq:def:fortin}
  b(\lambda_h,v-\Pi_Uv) = 0, \qquad \|\Pi_U v\|_{U,\varepsilon} \leq C_U \|v\|_{U,\varepsilon} 
  \qquad\forall \lambda_h\in \Lambda_h, v\in U.
\end{align}
The discrete version of~\eqref{eq:dglm} reads as follows: Find $(u_h,\lambda_h)\in U_h\times \Lambda_h$ such that
\begin{subequations}\label{eq:dglm:disc}
\begin{alignat}{2}
  &a(u_h,v_h) + b(\lambda_h,v_h) &\,=\,& \ip{f}{v_h}_\Omega, \label{eq:dglm:disc:a}
  \\
  &b(\mu_h,u_h) &\,=\,&0 \label{eq:dglm:disc:b}
\end{alignat}
\end{subequations}
for all $(v_h,\mu_h)\in U_h\times \Lambda_h$.
We call this numerical scheme the \emph{primal hybrid finite element method} (PHFEM).

The following result is immediate by the theory of mixed methods and Fortin operators, see~\cite{BoffiBrezziFortin}. 
\begin{theorem}\label{thm:discrete}
  Let $f\in L_2(\Omega)$. Suppose there exists a Fortin operator~\eqref{eq:def:fortin}. Then, problem~\eqref{eq:dglm:disc} admits a unique solution $(u_h,\lambda_h)\in U_h\times \Lambda_h$. Let $(u,\lambda) \in U\times \Lambda$ denote the solution of~\eqref{eq:dglm}. The quasi-optimality estimate
  \begin{align*}
    \|u-u_h\|_{U,\varepsilon} + \|\lambda-\lambda_h\|_{\Lambda,\varepsilon} \leq C\min_{(v,\mu)\in U_h\times \Lambda_h} \|u-v\|_{U,\varepsilon} + \|\lambda-\mu\|_{\Lambda,\varepsilon}
  \end{align*}
  holds, where $C>0$ depends on $C_U$ from~\eqref{eq:def:fortin}.
  \qed
\end{theorem}

\subsection{Dual hybrid variational formulation}\label{sec:dualmixedLM}
In this section we derive an alternative formulation based on the dual variable of model problem~\eqref{eq:model}. Consider the following first-order reformulation of~\eqref{eq:model}, 
\begin{align*}
  \bsigma - \varepsilon\nabla u &= 0, \\
  -\varepsilon\div\bsigma + u &= f. 
\end{align*}
Testing the first equation on $T\in\cT$ with $\btau\in \Hdivset{T}$ and the second with $v\in L^2(T)$ we find that (formally)
\begin{align*}
  \ip{\bsigma}{\btau}_T + \ip{u}{\varepsilon\div\bsigma}_T - \varepsilon\dual{\bsigma\cdot\normal_T}{u}_{\partial T}&= 0, \\
  \ip{-\varepsilon\div\bsigma}v_T + \ip{u}v_T &= \ip{f}v_T.
\end{align*}
Note that the second equation can be resolved as $u = \varepsilon\div\bsigma + f$ on $T$. Putting the latter identity into the first equation gives
\begin{align*}
  \ip{\bsigma}{\btau}_T + \varepsilon^2\ip{\div\bsigma}{\btau}_T - \varepsilon\dual{\bsigma\cdot\normal_T}{u}_{\partial T} &= \ip{f}{-\varepsilon\div\btau}_T.
\end{align*}
Summing over all $T\in\cT$ and introducing the novel variable $w = \trnabla_{\cT} u$ we end up with the following variational formulation: 
Find $(\bsigma,w) \in \Sigma \times W$ such that
\begin{subequations}\label{eq:dualdglm}
  \begin{alignat}{2}
    &a^\mathrm{dual}(\bsigma,\btau) + b^\mathrm{dual}(w,\btau) &\,=\,& \ip{f}{-\varepsilon\pwdiv\btau}_\Omega, \label{eq:dualdglm:a} \\
    &b^\mathrm{dual}(v,\bsigma) &\, = \, & 0 \label{eq:dualdglm:b}
  \end{alignat}
\end{subequations}
for all $(\btau,v)\in \Sigma\times W$ where the bilinear forms $a^\mathrm{dual}\colon \Sigma\times \Sigma \to \R$, $b^\mathrm{dual}\colon W\times \Sigma \to \R$ are given as
\begin{align*}
  a^\mathrm{dual}(\bsigma,\btau) &= \varepsilon^2\ip{\pwdiv\bsigma}{\pwdiv\btau}_\Omega + \ip{\bsigma}{\btau}_\Omega, \\
  b^\mathrm{dual}(w,\btau) &= -\varepsilon\dual{\btau\cdot\normal}{w}_{\partial\cT}. 
\end{align*}
We recall that the homogeneous boundary conditions are included in space $W$. Furthermore, we stress that~\eqref{eq:dualdglm:b} enforces continuity of $\bsigma$.

First, we prove that formulation~\eqref{eq:dualdglm} is equivalent to model problem~\eqref{eq:model}.
\begin{proposition}
  Let $f\in L_2(\Omega)$. Problems~\eqref{eq:model} and~\eqref{eq:dualdglm} are equivalent in the following sense: If $u\in H_0^1(\Omega)$ solves~\eqref{eq:model} then $(\varepsilon\nabla u,\trnabla_{\cT}u)\in \Sigma\times W$ solves~\eqref{eq:dualdglm}. 
  Conversely, if $(\bsigma,w)\in \Sigma\times W$ solves~\eqref{eq:dualdglm} then $\bsigma\in\Hdivset\Omega$ and $u:=\varepsilon\div\bsigma+f\in H_0^1(\Omega)$ solves~\eqref{eq:model}. Furthermore, $w = \trnabla_{\cT}u$ and $\bsigma = \varepsilon\nabla u$.
\end{proposition}
\begin{proof}
  Let $u\in H_0^1(\Omega)$ be the solution of~\eqref{eq:model}. Set $\bsigma = \varepsilon\nabla u$ and $w = \trnabla_\cT u$.
  By Lemma~\ref{lem:cont} we have that $b^\mathrm{dual}(v,\bsigma) = 0$ for all $v\in W$ which is~\eqref{eq:dualdglm:b}.
  Testing $\bsigma-\varepsilon\nabla u =0$ with some $\btau\in\Sigma$, integrating by parts and replacing $u$ by $\varepsilon\div\bsigma + f$ shows~\eqref{eq:dualdglm:a}.

  Conversely, let $(\bsigma,w)\in \Sigma\times W$ denote the solution of~\eqref{eq:dualdglm}. 
  By~\eqref{eq:dualdglm:b} and Lemma~\ref{lem:cont} we get that $\bsigma\in\Hdivset\Omega$. Taking $\btau\in \Hdivset\Omega$ in~\eqref{eq:dualdglm:a} and using Lemma~\ref{lem:cont} to conclude that $b^\mathrm{dual}(w,\btau) = 0$, we find that
  \begin{align*}
    \ip{\bsigma}{\btau}_\Omega + \varepsilon^2 \ip{\div\bsigma}{\div\btau}_\Omega = \ip{f}{-\varepsilon\div\btau}_\Omega \quad\forall \btau\in \Hdivset\Omega.
  \end{align*}
  Define $u = \varepsilon\div\bsigma+f$. Then, 
  \begin{align*}
    \ip{\bsigma}{\btau}_\Omega + \ip{u}{\varepsilon\div\btau}_\Omega &= 0, \\
    -\ip{\varepsilon\div\bsigma}{v}_\Omega + \ip{u}v_\Omega &= \ip{f}v_\Omega
  \end{align*}
  for all $(\btau,v)\in \Hdivset\Omega\times L^2(\Omega)$. This is the variational formulation of the first-order reformulation of~\eqref{eq:model}. 
  We conclude that $\bsigma=\varepsilon\nabla u$, $u\in H_0^1(\Omega)$ and $-\varepsilon^2\Delta u + u = f$. 
  It remains to prove that $w = \trnabla_{\cT}u$. This follows by integration by parts, which finishes the proof. 
\end{proof}

In what follows we analyze problem~\eqref{eq:dualdglm}. As for problem~\eqref{eq:dglm} we can apply the classic theory on mixed methods from~\cite{BoffiBrezziFortin}. 
\begin{theorem}
  Problem~\eqref{eq:dualdglm} is well posed. In particular, for any $f\in L_2(\Omega)$ there exists a unique solution $(\bsigma,w)\in \Sigma\times W$ with
  \begin{align*}
    \|\bsigma\|_{\Sigma,\varepsilon} + \|w\|_{W,\varepsilon} \lesssim \|f\|_\Omega. 
  \end{align*}
  The involved constant is independent of $\varepsilon$ and $\cT$. 
\end{theorem}
\begin{proof}
  The proof follows similar arguments as the one for Theorem~\ref{thm:dglm} and is left to the reader. 
\end{proof}

\subsection{DHFEM}\label{sec:dualdglm}
Let $\Sigma_h\subset \Sigma$, $W_h\subset W$ denote finite-dimensional subspaces. We say that $\Pi_\Sigma\colon \Sigma\to \Sigma_h$ is a Fortin operator if there exists $C_\Sigma>0$ with
\begin{align}\label{eq:def:fortin:dual}
  b^\mathrm{dual}(w_h,\btau-\Pi_\Sigma\btau) = 0, \qquad \|\Pi_\Sigma \btau\|_{\Sigma,\varepsilon} \leq C_\Sigma \|\btau\|_{\Sigma,\varepsilon} 
  \qquad\forall w_h\in W_h, \btau\in \Sigma.
\end{align}
The discrete version of~\eqref{eq:dualdglm} reads as follows: Find $(\bsigma_h,w_h)\in \Sigma_h\times W_h$ such that
\begin{subequations}\label{eq:dualdglm:disc}
\begin{alignat}{2}
  &a^\mathrm{dual}(\bsigma_h,\btau_h) + b^\mathrm{dual}(w_h,\btau_h) &\,=\,& \ip{f}{-\varepsilon\pwdiv\btau_h}_\Omega, \label{eq:dualdglm:disc:a}
  \\
  &b^\mathrm{dual}(v_h,\bsigma_h) &\,=\,&0 \label{eq:dualdglm:disc:b}
\end{alignat}
\end{subequations}
for all $(\btau_h,v_h)\in \Sigma_h\times W_h$.
We call this numerical scheme the \emph{dual hybrid finite element method} (DHFEM).

The following result is immediate by the theory of mixed methods and Fortin operators, see~\cite{BoffiBrezziFortin}. 
\begin{theorem}\label{thm:discrete:dual}
  Let $f\in L_2(\Omega)$. Suppose there exists a Fortin operator~\eqref{eq:def:fortin:dual}. Then, problem~\eqref{eq:dualdglm:disc} admits a unique solution $(\bsigma_h,w_h)\in \Sigma_h\times W_h$. Let $(\bsigma,w) \in \Sigma\times W$ denote the solution of~\eqref{eq:dualdglm}. The quasi-optimality estimate
  \begin{align*}
    \|\bsigma-\bsigma_h\|_{\Sigma,\varepsilon} + \|w-w_h\|_{W,\varepsilon} \leq C\min_{(\btau,v)\in \Sigma_h\times W_h} \|\bsigma-\btau\|_{\Sigma,\varepsilon} 
    + \|w-v\|_{W,\varepsilon}
  \end{align*}
  holds, where $C>0$ depends on $C_\Sigma$ from~\eqref{eq:def:fortin:dual}.
  \qed
\end{theorem}

Given the solution $(\bsigma_h,w_h)$ of~\eqref{eq:dualdglm:disc} we define an approximation of the primal variable $u$ by 
\begin{align}\label{eq:def:uhdual}
  u_h^\mathrm{dual} = \varepsilon\pwdiv\bsigma_h + f.
\end{align}
Since $u = \varepsilon\div\bsigma + f$ with $\bsigma=\varepsilon\nabla u$ and $u\in H_0^1(\Omega)$ being the solution of~\eqref{eq:model}, it is evident that
\begin{align*}
  \|u-u_h^\mathrm{dual}\|_\Omega \leq \varepsilon\|\div\bsigma-\pwdiv\bsigma_h\|_\Omega \leq \|\bsigma-\bsigma_h\|_{\Sigma,\varepsilon}.
\end{align*}

\subsection{Reducing degrees of freedom and parallel computations}
Recall that $a(\cdot,\cdot)$ is an inner product on the product space $U = H^1(\cT)$. 
Thus, $a(\cdot,\cdot)$ can be locally inverted on $U_{h}$ assuming that $U_h$ has product space structure, i.e,  $U_h = \prod_{T\in\cT} U_h(T)$.  Indeed,  let us define
\begin{equation}
a_T(u,v) = \epsilon^2(\nabla u,\nabla v)_T + (u,v)_T \quad \forall u,v\in H^1(T)\,.
\end{equation}
Then,  $a(u,v) = \sum_{T\in\cT} a_T(u,v) $ and,  from \eqref{eq:dglm:disc:a},  we arrive at
\begin{equation}
	a_T(u_h,v_h) = \varepsilon\, \langle\lambda_h,v_h\rangle_{\partial T} + (f,v_h)_{T}\,.
\end{equation}
Thus,  taking advantage of the bilinearity of $a_T(\cdot,\cdot)$ we propose the decomposition 
\begin{equation}
\label{uh-decom}
u_h|_T = \sum_{i=1}^{N_T} c_i^T\,\xi_i^T  + u_h^f|_T
\end{equation}
where  $c_i^T$  for $i=1,\ldots,N_T$ are constants to be determined
and,  $\xi_i^T,  u_h^f|_T\in U_h(T)$ respectively solve the problems
\begin{equation}
\label{local-problems}
a_T(\xi_i^T,v_h) =\varepsilon\, \langle\psi_i^T,v_h\rangle_{\partial T},
\text{ and,  }
a_T(u_h^f,v_h) = (f,v_h)_{T}\quad
\text{ for all } v_h \in U_h(T)
\end{equation}
and where $\{\psi_1^T,\ldots,\psi_{N_T}^T\}$ collects the elements of a basis
$\{\psi_1,\ldots,\psi_{\dim(\Lambda_h)}\}$
 for $\Lambda_h$ 
with support on $\partial T$.
On the other hand,   \eqref{eq:dglm:disc:a} can be used to define the skeleton problem:  Find $c_1,\ldots,c_{\dim(\Lambda_h)}$ such that
\begin{equation}
\sum_{i=1}^{\dim(\Lambda_h)}  c_i\,b(\psi_j,\xi_i) = -\,b(\psi_j,u_h^f)\,,
\text{ for all } j\in\{1,\ldots,\dim(\Lambda_h)\}\,,
\end{equation}
where each $c_i^T$ corresponds to some $c_j$ trough the same mapping 
that associate each $\psi_i^T$ to a $\psi_j$.  Therefore,  \eqref{uh-decom}
provides a local characterization of $u_h$ after an \emph{embarrassingly parallel}
process (in terms of $T\in\cT$)  related to solve local problems \eqref{local-problems}
and then to solve the reduced global problem with a number of degrees equal to $\dim(\Lambda_h)$.  
The process described above has been studied for other methods, e.g., multiscale methods~\cite{MR3336297}.

Similar considerations are valid for the DHFEM where the number of degrees of freedom reduces to $\dim(W_h)$.

\subsection{Formal comparison between the two methods}
If we choose polynomial spaces for the Lagrange multiplier space $\Lambda_h$ resp. $W_h$, e.g., 
\begin{align*}
  \Lambda_h = P^p(\partial\cT) \quad\text{resp.}\quad
  W_h = \trnabla_{\cT}(P^{p+1}(\cT)\cap H_0^1(\Omega))
\end{align*}
we have in general $\dim(\Lambda_h)>\dim(W_h)$ (see Section~\ref{sec:stability} below for precise definitions of the spaces). 
If $d=2$ and $p=0$, then $\dim(\Lambda_h)$ is equal to the number of edges of the triangulation $\cT$ and $\dim(W_h)$ is equal to the number of interior vertices of $\cT$.
From that point of view the DHFEM seems to be more attractive. Nevertheless, the DHFEM requires to discretize the vector-valued function space $\Sigma$. Furthermore, if approximations of the primal variable are of interest, then one needs an additional postprocessing step for DHFEM, see~\eqref{eq:def:uhdual}.

\section{Uniform stability of discrete methods}\label{sec:stability}
In this section we discuss and analyze certain properties of the discretization spaces that are sufficient to obtain uniform stable numerical schemes with respect to parameter $\varepsilon$ and mesh-size $h$. 
In view of the analysis of the methods from the previous section this means that given space $\Lambda_h$ we want to define $U_h$ so that~\eqref{eq:def:fortin} is satisfied (and similar for the DHFEM). 
First, we fix a discretization for $\Lambda$ and then discuss properties of $U_h$ to ensure uniform stability. Next, we fix a discretization for $W_h$ and discuss properties of $\Sigma_h$ that are sufficient for the existence of a Fortin operator with~\eqref{eq:def:fortin:dual}. 

For the discretization let $\cT$ denote a mesh of $\Omega$ into shape-regular simplices (that is, triangles for $d=2$ and tetrahedrons for $d=3$ such that $\max_{T\in\cT}\frac{h_T^d}{|T|}$ is bounded by a constant of order one).
We denote local and global mesh-size by $h_T = \diam(T)$ and $h = \max\{h_T\,:\,T\in\cT\}$, respectively. We choose the space
\begin{align}\label{eq:choice:Lambda}
  \Lambda_h = P^p(\partial\cT) = \trdiv_{\cT}(\RT^p(\cT))
\end{align}
where $\RT^p(\cT)$ denotes the Raviart--Thomas space of degree $p\in \N_0$, i.e., 
\begin{align*}
  \RT^p(\cT) = \set{\btau\in \Hdivset\Omega}{\btau|_T\in \RT^p(T), \, T\in\cT}, \quad \RT^p(T) = P^p(T)^d + \bx P^p(T). 
\end{align*}
Elements of $\Lambda_h$ are piecewise polynomials of degree $\leq p$ on the facets of $T\in\cT$.

The tricky part is to choose $U_h$ so that existence of a Fortin operator with~\eqref{eq:def:fortin} can be guaranteed. Indeed, choosing $U_h$ to be a purely polynomial space yields a constant $C_U \approx \max\{1,\sqrt{h_T\varepsilon^{-1}}\}$ as has been elaborated in~\cite[Theorem~13]{FuehrerHeuerFortin}.
Particularly, for coarse meshes $\varepsilon\ll h$ this means that $\inf$--$\sup$ constants can be small leading to large constants in Theorem~\ref{thm:discrete}.
To remove the dependence of $C_U$ on the ratio $\varepsilon/h_T$ the recent work~\cite{FuehrerHeuerFortin} suggests to include face bubble functions with exponential decay depending on the ratio $\varepsilon/h_T$ in the definition of space $U_h$.

For problem~\eqref{eq:dualdglm:disc} we choose
\begin{align}\label{eq:choice:W}
  W_h = \trnabla_{\cT}(P^{p+1}(\cT)\cap H_0^1(\Omega)).
\end{align}
Here, $P^p(\cT)$ denotes the space of $\cT$-piecewise polynomials of degree $\leq p$. Note that $W_h$ are traces of the standard Lagrange finite element spaces.

The remainder of this section is organized as follows: First, we recall the definition of the modified face bubble functions from~\cite{FuehrerHeuerFortin} in Section~\ref{sec:expbubble}. Then, in Section~\ref{sec:uFortin} resp.~\ref{sec:sigmaFortin} we discuss sufficient conditions for the existence of Fortin operators with~\eqref{eq:def:fortin} resp.~\eqref{eq:def:fortin:dual}. 
Finally, in Section~\ref{sec:expbubble:alt} we present another definition of modified face bubble functions that do not involve exponential functions.

\subsection{Face bubble functions with exponential layer}\label{sec:expbubble}
Given $T\in\cT$ let $\cF_T$ denote its facets (edges for $d=2$ resp. faces for $d=3$) and $\cV_T$ is the set of vertices of $T$.
For each $F\in\cF_T$ let $z_F^T\in \cV_T$ be the vertex opposite to $F$. We define the face bubble function by
\begin{align*}
  \eta_F^T = \prod_{z\in \cV_T\setminus\{z_F^T\}} \eta_z^T
\end{align*}
where $\eta_z^T\in P^1(T)$ is the barycentric coordinate function (i.e., hat function corresponding to a node $z$).
Further set $d_{F}^T = \eta_{z_F^T}^T$.
We use $P^p(T)$ to denote the space of polynomials over $T\in\cT$ of degree $\leq p\in\N_0$. Similarly, $P^p(F)$ is the space of polynomials over the facet $F\in\cF_T$.

We follow~\cite{FuehrerHeuerFortin} and modify $\eta_F^T$ by multiplying with an exponential layer depending on the ratio $h_T/\varepsilon$, 
\begin{align*}
  \eta_{F,\varepsilon}^T = \exp\big(-(h_T/\varepsilon) d_F^T\big) \eta_F^T.
\end{align*}
The function $\eta^T_{F,\varepsilon}$ is visualized in Figure~\ref{fig:bubbles} for the element $T$ spanned by the vertices $(0,0),(1,0),(0,1)$ and $F = (0,1)\times \{0\}$.
The next result summarizes some key properties of $\eta^T_{F,\varepsilon}$. A proof is given in~\cite[Lemma~6]{FuehrerHeuerFortin}.
\begin{lemma}\label{lem:propEtaFe}
  Let $F\in\cF_T$. Then, $\eta^T_F|_{\partial T} = \eta^T_{F,\varepsilon}|_{\partial T}$.
  \begin{align*}
    \|\eta^T_{F,\varepsilon}\|_T \lesssim |T|^{1/2}\left(\frac{\varepsilon}{h_T}\right)^{1/2}, \qquad
    \|\nabla\eta^T_{F,\varepsilon}\|_T \lesssim \frac{|T|^{1/2}}{h_T} \left(\frac{h_T}\varepsilon\right)^{1/2}.
  \end{align*}
  \qed
\end{lemma}

\begin{figure}
  \begin{center}
    \includegraphics[width=\textwidth]{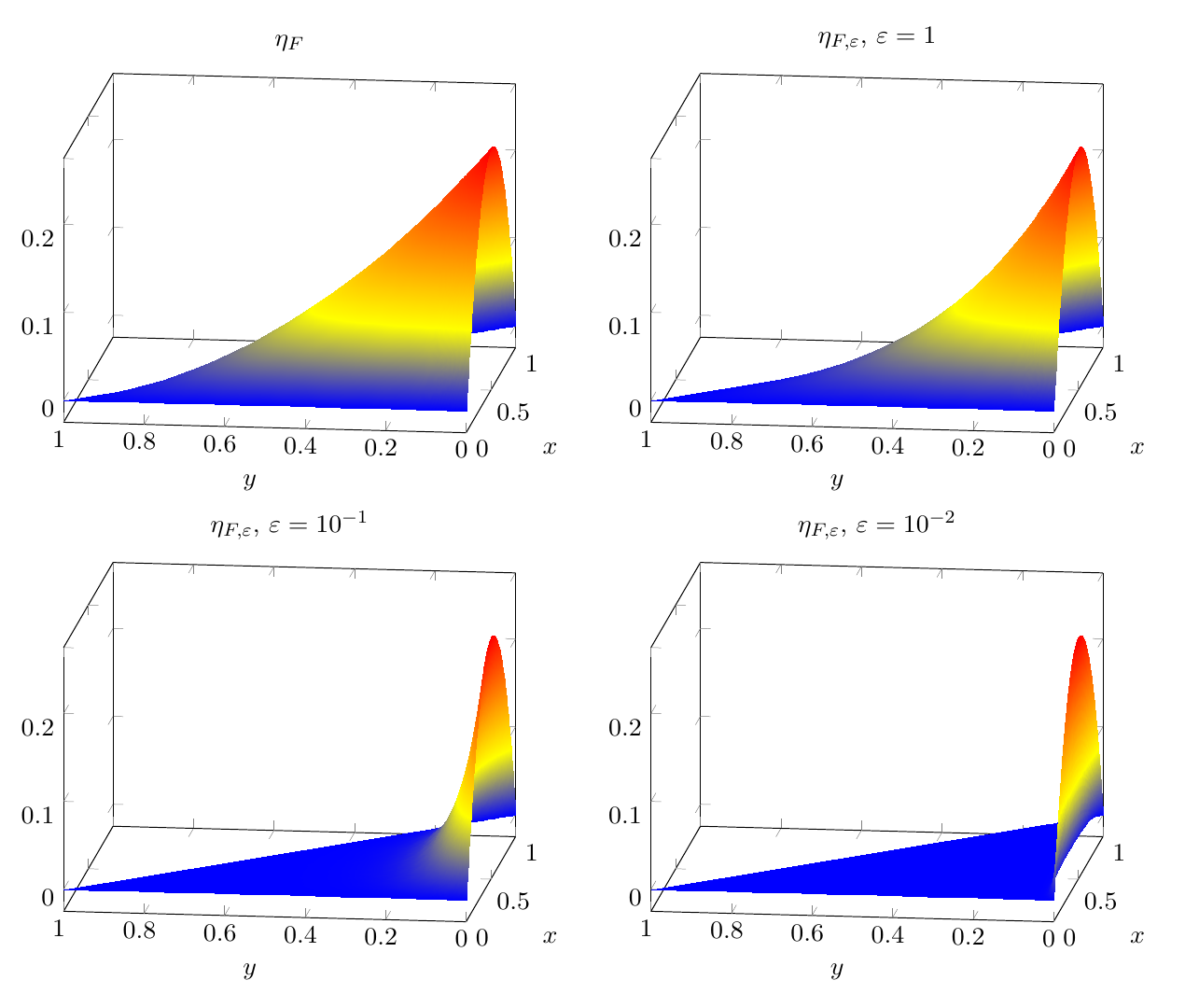}
  \end{center}
  \caption{Visualization of the exponential layer functions for different values of $\varepsilon$.}
  \label{fig:bubbles}
\end{figure}

To define higher order basis functions we follow~\cite[Section~2.4]{FuehrerHeuerFortin} but note that we use a slightly different notation here.
Let $\chi_{F,j}^T\in P^p(T)$, $j=1,\cdots,\dim(P^p(F))$ be such that the set $\{\chi_{F,j}^T|_F\,:\,F\in\cF_T\}$ forms a basis of $P^p(F)$.
Set 
\begin{align*}
  \eta_{F,j,\varepsilon}^T = \eta^T_{F,\varepsilon} \chi^T_{F,j}, \quad j=1,\ldots,\dim(P^p(F)), F\in \cF_T.
\end{align*}
We shall also define $\eta_{F,j}^T = \eta_F^T\chi^T_{F,j}$, $j=1,\cdots,\dim(P^p(F))$, $F\in\cF_T$.

Let $P_c^{p+1}(\cF_T)$ denote the trace space of $P^{p+1}(T)$. Let $\nu_{\partial T,j}\in P^{p+1}(T)$, $j=1,\dots,\dim(P_c^{p+1}(\cF_T))$ be such that $\set{\nu_{\partial T,j}|_{\partial T}}{j=1,\cdots,\dim(P_c^{p+1}(\cF_T))}$ denotes a basis of $P_c^{p+1}(\cF_T)$.
Set $\eeta_F^T = \eta_F^T\normal_T|_F$, $\eeta_{\partial T} = \sum_{F\in\cF_T}\eeta_F^T$ and $\eeta_{F,\varepsilon}^T = \eta_{F,\varepsilon}^T\normal_T|_F$, $\eeta_{\partial T,\varepsilon} = \sum_{F\in\cF_T}\eeta_{F,\varepsilon}^T$.
Further, define
\begin{align*}
  \eeta_{\partial T,j} = \eeta_{\partial T} \nu_{\partial T,j}, \quad \eeta_{\partial T,\varepsilon,j} = \eeta_{\partial T,\varepsilon} \nu_{\partial T,\varepsilon,j}
\end{align*}
for $j=1,\cdots,\dim(P_c^{p+1}(\cF_T))$.

\subsection{Conditions on $U_h$ and Fortin operator}\label{sec:uFortin}
Given $p\in \N_0$ we set for each $T\in\cT$, 
\begin{align*}
  \widetilde U_{hp,\varepsilon}(T) = \begin{cases}
    P^0(T) + \linhull\set{\eta_{F,j,\varepsilon}^T}{j=1,\ldots,\dim(P^p(F)), \, F\in \cF_T} & \varepsilon <h_T, \\
    P^0(T) + \linhull\set{\eta_{F,j}^T}{j=1,\ldots,\dim(P^p(F)), \, F\in \cF_T} & \varepsilon \geq h_T.
  \end{cases}
\end{align*}
Then, we define
\begin{align*}
  \widetilde U_{hp,\varepsilon} = \prod_{T\in\cT} \widetilde U_{hp,\varepsilon}(T).
\end{align*}

The following result follows by combining~\cite[Lemma~3.1 and Lemma~3.6]{FuehrerHeuerFortin}. 
\begin{theorem}
  With $\Lambda_h$ defined in~\eqref{eq:choice:Lambda}, suppose that $U_h\subset U$ is a finite-dimensional space. 
  If $\widetilde U_{hp,\varepsilon}\subseteq U_h$ then there exists an operator $\Pi_U\colon U\to U_h$ satisfying~\eqref{eq:def:fortin} with constant $C_U$ independent of the mesh-size and $\varepsilon$.

  In particular, the assumptions from Theorem~\ref{thm:discrete} are satisfied.
  \qed
\end{theorem}

\subsection{Conditions on $\Sigma_h$ and Fortin operator}\label{sec:sigmaFortin}
Given $p\in \N_0$ we set for each $T\in\cT$, 
\begin{align*}
  \widetilde \Sigma_{hp,\varepsilon}(T) = \begin{cases}
    P^0(T)^d + \linhull\set{\eeta_{\partial T,j},\eeta_{\partial T,\varepsilon,j}}{j=1,\ldots,\dim(P^{p+1}_c(\cF_T))} & \varepsilon <h_T, \\
    P^0(T)^d + \linhull\set{\eeta_{\partial T,j}}{j=1,\ldots,\dim(P^{p+1}_c(\cF_T))} & \varepsilon \geq h_T.
  \end{cases}
\end{align*}
Then, we define
\begin{align*}
  \widetilde \Sigma_{hp,\varepsilon} = \prod_{T\in\cT} \widetilde \Sigma_{hp,\varepsilon}(T).
\end{align*}
We note that our definition of the spaces slightly deviates from~\cite[Section~4]{FuehrerHeuerFortin}. There, instead of the functions $\eeta_{\partial T,j}$, Raviart--Thomas basis functions are used. However, the same argumentations as in the proofs of~\cite[Lemma~4.1 and Lemma~4.5]{FuehrerHeuerFortin} can be used to show the next result.
\begin{theorem}\label{thm:fortin:dual}
  With $W_h$ defined in~\eqref{eq:choice:W}, suppose that $\Sigma_h\subset \Sigma$ is a finite-dimensional space. 
  If $\widetilde \Sigma_{hp,\varepsilon}\subseteq \Sigma_h$ then there exists an operator $\Pi_\Sigma\colon \Sigma\to \Sigma_h$ satisfying~\eqref{eq:def:fortin:dual}.

  In particular, the assumptions from Theorem~\ref{thm:discrete:dual} are satisfied.
  \qed
\end{theorem}

In~\cite[Section~4.3 and~4.4]{FuehrerHeuerFortin} some lowest-order cases ($p=0$) are considered. Define
\begin{align*}
  \widetilde \Sigma_{h,\varepsilon}(T) = \begin{cases}
    P^0(T)^d + \linhull\set{\eeta_{F,\varepsilon}}{F\in \cF_T} & \varepsilon <h_T, \\
    P^0(T)^d + \linhull\set{\eeta_{F}}{\, F\in \cF_T} & \varepsilon \geq h_T.
  \end{cases}
\end{align*}
The following result follows from the analysis given in~\cite[Section~4.3 and~4.4]{FuehrerHeuerFortin}.
\begin{theorem}
  Let $p=0$. The assertion of Theorem~\ref{thm:fortin:dual} hold true if $\widetilde \Sigma_{hp,\varepsilon}$ is replaced by $\widetilde \Sigma_{h,\varepsilon}: = \prod_{T\in\cT} \widetilde \Sigma_{h,\varepsilon}(T)$.
\end{theorem}

\subsection{Alternative for modified face bubbles}\label{sec:expbubble:alt}
In this section we present an alternative to the face bubble functions with exponential layer introduced in Section~\ref{sec:expbubble} above. 
For $T\in\cT$ and $(x,y)\in T$, $F\in\cF_T$ set
\begin{align*}
  \widetilde\eta_{F,\varepsilon}^T(x,y) = \begin{cases}
    \eta_F^T(x,y) \left(1-\frac{h_T}\varepsilon d_F^T(x,y)\right) & \text{if }d_F^T(x,y)<\tfrac\varepsilon{h_T}, \\
    0 & \text{if }d_F^T(x,y)\geq \tfrac\varepsilon{h_T}.
  \end{cases}
\end{align*}
Note that $\widetilde\eta_{F,\varepsilon}^T$ is a piecewise polynomial and the term $1-\frac{h_T}\varepsilon d_F^T(x,y)$ can be seen as a first-order approximation of the exponential function $\exp(-h_T/\varepsilon d_F^T(x,y))$.
Figure~\ref{fig:bubblesNew} visualizes differences between $\widetilde\eta_{F,\varepsilon}^T$ and $\eta_{F,\varepsilon}^T$ for some values of $\varepsilon$ on the element $T$ with vertices $(0,0)$, $(1,0)$, $(0,1)$ and $F = (0,1)\times \{0\}$.

\begin{figure}
  \begin{center}
    \includegraphics[width=\textwidth]{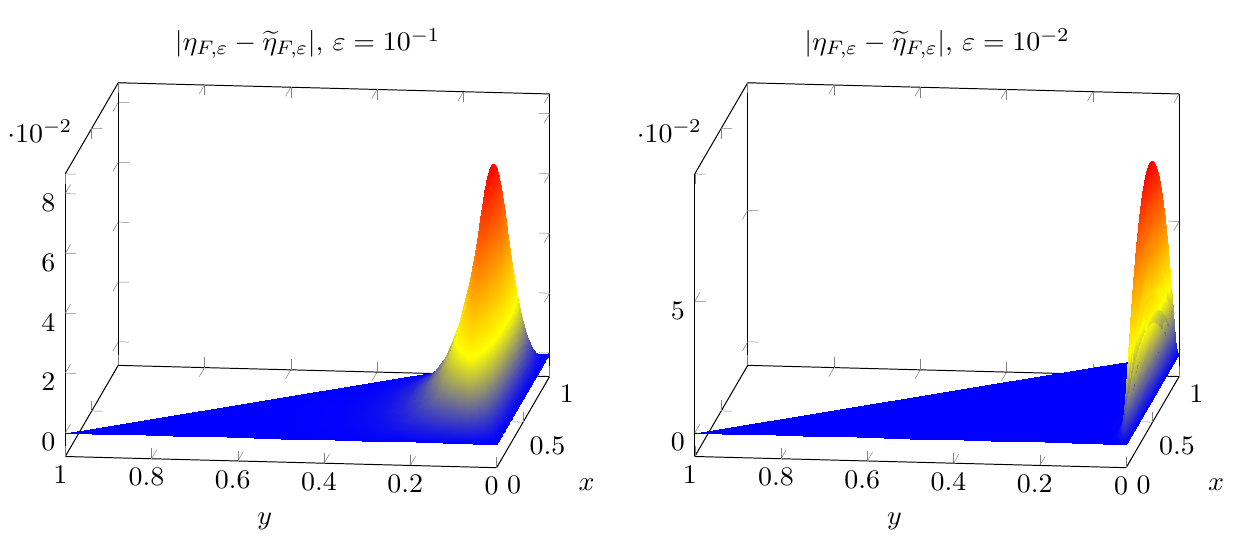}
  \end{center}
  \caption{Comparison of face bubble functions with exponential resp. polynomial layers defined in Section~\ref{sec:expbubble} resp.~\ref{sec:expbubble:alt} for different values of $\varepsilon$.}
  \label{fig:bubblesNew}
\end{figure}

The next result is essential for the analysis and states that $\widetilde\eta_{F,\varepsilon}^T$ has the same (scaling) properties as $\eta_{F,\varepsilon}^T$, cf. Lemma~\ref{lem:propEtaFe}. 
\begin{lemma}\label{ref:lem:propEtaFe:alt}
  Suppose that $T\in\cT$, $F\in\cF_T$. Then, $\widetilde\eta_{F,\varepsilon}^T\in H^1(T)$ and $\widetilde\eta_{F,\varepsilon}^T|_{\partial T} = \eta_F^T|_{\partial T}$. If $\varepsilon\lesssim h_T$ then
  \begin{align*}
    \|\widetilde\eta^T_{F,\varepsilon}\|_T \lesssim |T|^{1/2}\left(\frac{\varepsilon}{h_T}\right)^{1/2}, \qquad
    \|\nabla\widetilde\eta^T_{F,\varepsilon}\|_T \lesssim \frac{|T|^{1/2}}{h_T} \left(\frac{h_T}\varepsilon\right)^{1/2}.
  \end{align*}
\end{lemma}
\begin{proof}
  Since $\widetilde\eta_{F,\varepsilon}^T$ is defined as elementwise polynomial and is continuous across the line given by $d_F^T(x,y) = \varepsilon/h_T$ we have that $\widetilde\eta_{F,\varepsilon}^T\in H^1(T)$.
  Then,  $\widetilde\eta_{F,\varepsilon}|_{\partial T} = \eta_F^T|_{\partial T}$ follows since $d_F^T(x,y)=0$ for $(x,y)\in F$ and $\eta_{F}^T|_{F'}=0$ for $F'\in\cF_T\setminus\{F\}$.

  Let $T_\varepsilon = \set{(x,y)}{d_F^T(x,y)<\varepsilon/h_T}$ and note that $|T_\varepsilon| \eqsim h_T^{d-1}\varepsilon$.
  The function $|\widetilde\eta_{F,\varepsilon}^T|$ is bounded by one, therefore, 
  \begin{align*}
    \|\widetilde\eta_{F,\varepsilon}^T\|_T^2 \leq \|1\|_{T_\varepsilon}^2 \eqsim h_T^{d-1}\varepsilon \eqsim |T| \frac{\varepsilon}{h_T}
  \end{align*}
  For the estimation of the norm of the gradient we apply the triangle inequality, product rule, scaling arguments and the fact that $\varepsilon\lesssim h_T$ to find that
  \begin{align*}
    \|\nabla\widetilde\eta_{F,\varepsilon}^T\|_T &\lesssim \|\nabla\eta_F^T\|_{T_\varepsilon} 
    + \frac{h_T}\varepsilon\|d_F^T\nabla \eta_F^T\|_{T_\varepsilon} 
    + \frac{h_T}\varepsilon\|\eta_{F}^T\nabla d_F^T\|_{T_\varepsilon}
    \\
    &\lesssim h_T^{-1}(h_T^{d-1}\varepsilon)^{1/2} + \frac{h_T}{\varepsilon} h_T^{-1} (h_T^{d-1}\varepsilon)^{1/2}
    \lesssim \left(\frac{h_T}{\varepsilon}\right)^{1/2} \frac{1}{h_T} |T|^{1/2}.
  \end{align*}
  This finishes the proof.
\end{proof}

Following the technical analysis from~\cite{FuehrerHeuerFortin} replacing the functions $\eta_{F,\varepsilon}^T$ by $\widetilde\eta_{F,\varepsilon}^T$ shows the next result.
\begin{corollary}
  All results from Section~\ref{sec:uFortin} and~\ref{sec:sigmaFortin} hold true with $\eta_{F,\varepsilon}^T$ replaced by $\widetilde\eta_{F,\varepsilon}^T$ in all definitions of Section~\ref{sec:expbubble}.
\end{corollary}

\section{A posteriori error analysis}\label{sec:apost}
In this section we derive some simple error estimators that can be used to steer adaptive algorithms. 
The estimators are based on an additional Fortin property. For the sake of a simpler presentation of the main arguments we restrict to some lowest-order discretizations in this section though we stress that the ideas can be applied for arbitrary order discretizations as well. 

\subsection{Auxiliary results}
In this section we collect some results that will be used in the analysis given in Section~\ref{sec:apost:dglm} and~\ref{sec:apost:dualdglm} below. 

Let $\Pi_T^p\colon L_2(T)\to P^p(T)$, $\Pi_\cT^p\colon L_2(\Omega)\to P^p(\cT)$, $\Pi_F^p\colon L_2(F)\to P^p(F)$ be the $L_2$ orthogonal projections and $h_\cT\in L_\infty(\Omega)$, $h_\cT|_T = h_T$ ($T\in\cT$) the mesh-width function.
With $\Omega_T\subseteq \Omega$, $\Omega_F\subseteq\Omega$ we denote neighborhoods of $T\in\cT$, resp. $F\in \cF$, 
\begin{align*}
  \Omega_T &= \mathrm{int}\left(\bigcup_{T'\in\cT, \,\overline{T}\cap\overline{T'}\neq \emptyset} \overline{T'}\right), \quad
  \Omega_F = \mathrm{int}\left(\bigcup_{T\in\cT, \,F\in\cF_T} \overline{T}\right).
\end{align*}
Here, $\mathrm{int}(\cdot)$ denotes the interior of a set and $\cF$ is the set of all facets of $\cT$.

We also use the trace inequality, cf.~\cite[Lemma~1.6.6]{BrennerScott},
\begin{align*}
  \|v\|_{\partial T} \lesssim \frac{1}{h_T^{1/2}} \|v\|_T^{1/2}(\|v\|_T+h_T\|\nabla v\|_T)^{1/2} \lesssim \frac{1}{h_T^{1/2}}\|v\|_T + h_T^{1/2}\|\nabla v\|_T \quad\forall v\in H^1(T).
\end{align*}
The first estimate also implies that
\begin{align*}
  \|(1-\Pi_F^0)v\|_F \leq \|(1-\Pi_T^0)v\|_{\partial T} \lesssim \|v\|_T^{1/2}\|\nabla v\|_T^{1/2}.
\end{align*}

Let $J_\cT\colon L_2(\Omega) \to P^1(\cT)\cap H^1(\Omega)$ denote a quasi-interpolation with the properties
\begin{align}\label{eq:propquasiint}
  \|J_{\cT}v\|_T\lesssim \|v\|_{\Omega_T}, \quad \|v-J_\cT v\|_T \lesssim h_T\|\nabla v\|_{\Omega_T}, \quad 
  \|\nabla J_\cT v\|_T \lesssim \|\nabla v\|_{\Omega_T} \quad\forall \, T\in\cT
\end{align}
and all $v\in H^1(\Omega)$.
An example of such an operator is the Cl\'ement quasi-interpolator~\cite{Clement75}.
With the trace inequality it follows that
\begin{align*}
  \|v-J_\cT v\|_{\partial T} \lesssim h_T^{1/2}\|\nabla v\|_{\Omega_T} \quad\forall v\in H^1(\Omega), \, T\in\cT.
\end{align*}
Furthermore, let $J_{\cT,0}\colon L_2(\Omega)\to P^1(\cT)\cap H_0^1(\Omega)$ denote a quasi-interpolation with the same properties as $J_\cT$ given above, i.e.,~\eqref{eq:propquasiint} holds for all $v\in H_0^1(\Omega)$. Again, we can use the Cl\'ement quasi-interpolator with vanishing boundary values.

Given $F\in\cF$ we denote by $\jump{v}|_F$ the jump of $v$ across the interface $F$. Similarly, $\jump{\btau\cdot\normal}|_F$ denotes the jump in normal direction and $\jump{\btau\times \normal}|_F$ the jump of the tangent component.
If $F$ is a boundary facet then the jump reduces to the trace on the boundary, i.e., $\jump{v}|_F = v|_F$.

\subsection{Error estimator for the PHFEM}\label{sec:apost:dglm}
Throughout this section we consider for each $T\in\cT$, 
\begin{align*}
  U_{h,\varepsilon}(T) = \begin{cases}
    P^1(T) + \linhull\set{\eta_{F,\varepsilon}^T}{F\in\cF_T} + \linhull\{\eta_T\} & \text{if } \varepsilon<h_T, \\
    P^1(T) + \linhull\set{\eta_{F}^T}{F\in\cF_T} + \linhull\{\eta_T\} & \text{if } \varepsilon\geq h_T, \\
  \end{cases}
\end{align*}
where $\eta_T = \prod_{z\in\cV_T} \eta_z^T$ is the element bubble function.
Then, define the product space
\begin{align*}
  U_h=U_{h,\varepsilon} = \prod_{T\in\cT} U_{h,\varepsilon}(T).
\end{align*}
We also consider space $\Lambda_h$ for $p=0$, i.e., $\Lambda_h = P^0(\partial\cT)$. Let us briefly comment on the choice of the local space $U_{h,\varepsilon}(T)$. We have included also polynomials of degree $1$ to ensure that gradients can be approximated (particularly, when $\varepsilon\gtrsim h_T$), the functions $\eta_{F,\varepsilon}^T$ resp. $\eta_{F}^T$ are included to ensure stability of the discrete scheme, cf. Section~\ref{sec:stability} and, finally, the element bubble function is included to ensure the additional orthogonality property given in Theorem~\ref{thm:fortin:aug} below. 
In particular, the next result follows by combining~\cite[Corollary~3.3 and Theorem~3.7]{FuehrerHeuerFortin}.
\begin{theorem}\label{thm:fortin:aug}
  There exists $\Pi_U\colon U\to U_{h}$ satisfying~\eqref{eq:def:fortin} with constant $C_U$ independent of $\varepsilon$ and $h$, and satisfying the additional property
  \begin{align*}
    \ip{1}{v-\Pi_Uv}_T = 0 \quad\forall T\in \cT, \,v\in U.
  \end{align*}
  \qed
\end{theorem}

For the remainder of this section let $(u_h,\lambda_h)\in U_{h}\times \Lambda_h$ denote the unique solution of~\eqref{eq:dglm:disc}.
Define for each $T\in\cT$ the (squared) indicator
\begin{align*}
  \rho(T)^2 &= \|(1-\Pi_T^0)(u_h-f)\|_T^2 + \varepsilon^2\|(1-\Pi_T^0)\nabla u_h\|_T^2 + \varepsilon\|\jump{u_h}\|_{\partial T}^2
  +\varepsilon^2 h_T \|\jump{\nabla u_h\times\normal}\|_{\partial T}^2.
\end{align*}
Note that $\rho(T)$ does not involve the approximation $\lambda_h$ of the Lagrange multiplier $\lambda$. 

In the next results we analyze reliability and efficiency of the estimator 
\begin{align*}
  \rho^2 = \sum_{T\in\cT} \rho(T)^2.
\end{align*}

\begin{theorem}\label{thm:dglm:apost}
  Let $(u,\lambda)\in U\times \Lambda$ and $(u_h,\lambda_h)\in U_h\times \Lambda_h$ denote the solutions of~\eqref{eq:dglm} and~\eqref{eq:dglm:disc}, respectively. 
  If $\Omega$ is convex, then estimator $\rho$ is reliable, i.e., there exists a constant $C_\mathrm{rel}$ independent of $\varepsilon$, $h$, $(u_h,\lambda_h)$ such that
  \begin{align*}
    \|u-u_h\|_{U,\varepsilon} + \|\lambda-\lambda_h\|_{\Lambda,\varepsilon} \leq C_\mathrm{rel} \rho.
  \end{align*}
\end{theorem}
\begin{proof}
  Recall that the system~\eqref{eq:dglm} is $\inf$--$\sup$ stable, and therefore, 
  \begin{align*}
    \|u-u_h\|_{U,\varepsilon} + \|\lambda-\lambda_h\|_{\Lambda,\varepsilon}
    \eqsim \sup_{0\neq (v,\mu)\in U\times \Lambda} \frac{a(u-u_h,v) + b(\lambda-\lambda_h,v) + b(\mu,u-u_h)}{\|(v,\mu)\|_{U\times \Lambda}}.
  \end{align*}
  Let $v\in U$ be arbitrary and let $\Pi_U$ denote the operator from Theorem~\ref{thm:fortin:aug}.
  Using~\eqref{eq:dglm:a}, \eqref{eq:dglm:disc:a}, $b(\lambda_h,v-\Pi_U v)=0$, and $\ip{1}{v-\Pi_Uv}_T = 0$ for all $T\in\cT$,  we obtain
  \begin{align*}
    a(u-u_h,v) + b(\lambda-\lambda_h,v) &= a(u-u_h,v-\Pi_U v) + b(\lambda-\lambda_h,v-\Pi_Uv)\\
    &= \ip{f-u_h}{v-\Pi_Uv}_\Omega -\varepsilon^2\ip{\pwnabla u_h}{\pwnabla(v-\Pi_Uv)}_\Omega
    \\
    &\lesssim \|(1-\Pi_\cT^0)(f-u_h)\|_\Omega\|v\|_{U,\varepsilon} + |\varepsilon^2\ip{\pwnabla u_h}{\pwnabla(v-\Pi_Uv)}_\Omega|.
  \end{align*}
  Let $\bq\in P^0(T)^d$. From the properties of the Fortin operator it follows that
  \begin{align*}
    \ip{\bq}{\pwnabla(v-\Pi_U v)}_T = \dual{\bq\cdot\normal_T}{v-\Pi_U v}_{\partial T} = 0.
  \end{align*}
  Consequently, 
  \begin{align*}
    |\varepsilon^2\ip{\pwnabla u_h}{\pwnabla(v-\Pi_Uv)}_\Omega| &= |\varepsilon^2\ip{(1-\Pi_\cT^0)\pwnabla u_h}{\pwnabla(v-\Pi_Uv)}_\Omega|
    \\
    &\lesssim \varepsilon\|(1-\Pi_\cT^0)\pwnabla u_h\|_\Omega \|v\|_{U,\varepsilon}.
  \end{align*}
  Putting all estimates obtained so far together we see that
  \begin{align*}
    \|u-u_h\|_{U,\varepsilon} + \|\lambda-\lambda_h\|_{\Lambda,\varepsilon} 
    &\lesssim \|(1-\Pi_\cT^0)(f-u_h)\|_\Omega + \varepsilon\|(1-\Pi_\cT^0)\pwnabla u_h\|_\Omega 
    \\
    &\qquad + \sup_{0\neq \mu\in \Lambda} \frac{b(\mu,u-u_h)}{\|\mu\|_{\Lambda,\varepsilon}}.
  \end{align*}
  It remains to localize the last term on the right-hand side. Since $u\in H_0^1(\Omega)$ we have that $b(\mu,u)=0$ for all $\mu\in \Lambda$ by Lemma~\ref{lem:cont}.
  This and~\eqref{eq:dglm:disc:b} give $b(\mu,u-u_h) = \varepsilon\dual{\mu-\mu_h}{u_h}_{\partial\cT}$ for all $\mu\in \Lambda$, $\mu_h\in\Lambda_h$.
  Let $\mu\in \Lambda$ be given. 
  There exists $\btau\in \Hdivset\Omega$ with $\trdiv_\cT\btau = \mu$ and $\|\btau\|_{\div,\varepsilon} = \|\mu\|_{\Lambda,\varepsilon}$. We consider the Helmholtz decomposition
  \begin{align*}
    \btau = \nabla r + \curl q
  \end{align*}
  with $r\in H_0^1(\Omega)$ and $\Delta r = \div\btau$. Note that by elliptic regularity we get $r\in H^2(\Omega)$ with $\|D^2 r\|_\Omega \lesssim \|\div\btau\|_\Omega$.
  Let $\Pi_\cT^{\div}\colon \Hdivset\Omega \to \RT^0(\cT)$ denote the Raviart--Thomas interpolator.
  Set $\btau_h = \Pi_\cT^\div\nabla r + \curl(J_\cT q)$. Note that $\curl(J_\cT q)\in \RT^0(\cT)$. Therefore, $\mu_h = \trdiv_\cT\btau_h \in \Lambda_h$ and
  \begin{align}\label{eq:apost:dglm:rel:1}
    \varepsilon \dual{\mu-\mu_h}{u_h}_{\partial\cT} = \varepsilon\dual{\trdiv_\cT(\nabla r-\Pi_\cT^{\div}\nabla r)}{u_h}_{\partial\cT} + \varepsilon\dual{\trdiv_\cT(\curl(1-J_\cT)q)}{u_h}_{\partial\cT}.
  \end{align}
  To estimate the first term on the right-hand side we note that $\nabla r|_{\partial T} \in L^2(\partial T)$ and $\Pi^\div\nabla r\cdot\normal|_F = \Pi_F^0\nabla r\cdot\normal|_F$ for facets $F$. It follows that
  \begin{align*}
    \varepsilon\dual{\trdiv_\cT(\nabla r-\Pi_\cT^{\div}\nabla r)}{u_h}_{\partial\cT} &= \sum_{F\in\cF} \varepsilon\dual{(1-\Pi_F^0)\nabla r\cdot\normal}{[u_h]}_{F}
    \\
    &\leq \sum_{F\in\cF} \varepsilon^{1/2} \|(1-\Pi_F^0)\nabla r\cdot\normal\|_F  \varepsilon^{1/2}\|[u_h]\|_F.
  \end{align*}
  Using the trace inequality we further obtain for any $T\in\cT$ with $F\in\cF_T$ that
  \begin{align*}
    \varepsilon^{1/2} \|(1-\Pi_F^0)\nabla r\cdot\normal\|_F \lesssim \|\nabla r\|_T^{1/2} \varepsilon^{1/2}\|D^2 r\|_T^{1/2} \lesssim \|\nabla r\|_T + \varepsilon\|D^2 r\|_T 
  \end{align*}
  Summing over all $F\in\cF$ yields
  \begin{align*}
    \varepsilon^{1/2} \|(1-\Pi_F^0)\nabla r\cdot\normal\|_F &\lesssim (\|\nabla r\|_\Omega + \varepsilon\|D^2 r\|_\Omega)\varepsilon^{1/2}\|[u_h]\|_{\cF} \\
    &\lesssim (\|\btau\|_\Omega + \varepsilon\|\div\btau\|_\Omega)\varepsilon^{1/2}\|[u_h]\|_{\cF},
  \end{align*}
  where $\|[u_h]\|_{\cF}^2 := \sum_{F\in\cF} \|[u_h]\|_F^2 \eqsim  \sum_{T\in\cT} \|[u_h]\|_{\partial T}^2 =: \|[u_h]\|_{\partial\cT}^2$.

For the second term on the right-hand side of~\eqref{eq:apost:dglm:rel:1} we use elementwise integration by parts to see that
\begin{align*}
  \varepsilon\dual{\trdiv_\cT(\curl(1-J_\cT)q)}{u_h}_{\partial\cT} &= \sum_{T\in\cT} \varepsilon\ip{\curl(1-J_\cT)q}{\nabla u_h}_T + \varepsilon\ip{\div\curl(1-J_\cT)q}{u_h}_T 
 \\
  &= \sum_{T\in\cT} \varepsilon\dual{(1-J_\cT)q}{\nabla u_h\times\normal}_{\partial T} = \sum_{F\in\cF} \varepsilon\dual{(1-J_\cT)q}{[\nabla u_h\times\normal]}_F.
\end{align*}
Using that $\|(1-J_\cT)q\|_F\lesssim h_T^{1/2}\|\nabla q\|_{\Omega_T}$ for all $T$ with $F\in\cF_T$ we further get that
\begin{align*}
  \varepsilon\dual{\trdiv_\cT(\curl(1-J_\cT)q)}{u_h}_{\partial\cT} &\lesssim \sum_{F\in\cF} \varepsilon\dual{(1-J_\cT)q}{[\nabla u_h\times\normal]}_F
  \lesssim \|\nabla q\|_\Omega\varepsilon\|h_\cT^{1/2}[\nabla u_h\times\normal]\|_{\partial\cT} \\
  &\lesssim \|\btau\|_{\div,\varepsilon} \varepsilon\|h_\cT^{1/2}[\nabla u_h\times\normal]\|_{\partial\cT}.
\end{align*}
We thus have shown that
\begin{align*}
  b(\mu-\mu_h,u_h) &\lesssim (\varepsilon^{1/2}\|[u_h]\|_{\partial\cT} + \varepsilon\|h_\cT^{1/2}[\nabla u_h\times\normal]\|_{\partial\cT})\|\btau\|_{\div,\varepsilon}
  \\
  &= (\varepsilon^{1/2}\|[u_h]\|_{\partial\cT} + \varepsilon\|h_\cT^{1/2}[\nabla u_h\times\normal]\|_{\partial\cT})\|\mu\|_{\Lambda,\varepsilon}
\end{align*}
for all $\mu\in \Lambda$ which finishes the proof. 
\end{proof}

An efficiency-type estimate follows from the next technical result. 
\begin{lemma}
  Let $(u,\lambda)$ and $(u_h,\lambda_h)$ denote the unique solutions of~\eqref{eq:dglm} and~\eqref{eq:dglm:disc}, respectively. 
  Then, for any $T\in\cT$ we have 
  \begin{align*}
    \varepsilon^{1/2}\|[u_h]\|_{\partial T} + \varepsilon h_T^{1/2} \|[\nabla u_h\times\normal]\|_{\partial T} \lesssim \|u-u_h\|_{\Omega_T} + \varepsilon\|\pwnabla(u-u_h)\|_{\Omega_T}.
  \end{align*}
\end{lemma}
\begin{proof}
  Recall that $0=b(\mu_h,u_h) = -\sum_{F\in\cF} \varepsilon\dual{\mu_h}{[u_h]}_F$ for all $\mu_h\in\Lambda_h$. This means that $[u_h]$ is orthogonal to constants on  each facet $F$. For each $F\in\cF$ there exist at most two elements, $T^\pm\in\cT$ which share the edge $F$. Then, using continuity $[u] = 0$,
  \begin{align*}
    \varepsilon^{1/2}\|[u_h]\|_F &= \varepsilon^{1/2}\|(1-\Pi_F^0)[u_h]\|_F 
    \\
    &\leq \varepsilon^{1/2} \|(1-\Pi_F^0)(u-u_h)|_{T^+}\|_F + \varepsilon^{1/2} \|(1-\Pi_F^0)(u-u_h)|_{T^-}\|_F.
  \end{align*}
  We apply the trace inequality to obtain that 
  \begin{align*}
    \varepsilon^{1/2} \|(1-\Pi_F^0)(u-u_h)|_{T^\pm}\|_F &\lesssim \|u-u_h\|_{T^\pm}^{1/2} \varepsilon^{1/2}\|\nabla(u-u_h)\|_{T^\pm}^{1/2} \\
    &\lesssim \|u-u_h\|_{T^\pm} + \varepsilon\|\nabla(u-u_h)\|_{T^\pm}.
  \end{align*}
  Therefore, 
  \begin{align*}
    \varepsilon^{1/2}\|[u_h]\|_{\partial T} \lesssim \|u-u_h\|_{\Omega_T} + \varepsilon\|\pwnabla(u-u_h)\|_{\Omega_T}.
  \end{align*}
  
  It remains to estimate term $\varepsilon h_T^{1/2} \|[\nabla u_h\times\normal]\|_{\partial T}$. For this we can use the usual bubble function technique by Verf\"urth, see, e.g.~\cite{Verf94}, which yields
  \begin{align*}
    \varepsilon h_T^{1/2} \|[\nabla u_h\times\normal]\|_{\partial T} \lesssim \varepsilon\|\pwnabla(u-u_h)\|_{\Omega_T}.
  \end{align*}
  This concludes the proof. 
\end{proof}

The next result shows an efficiency-type estimate, i.e., efficiency up to local best approximation and data oscillation terms. Its proof follows directly from the last result and the triangle inequality. 
\begin{corollary}
  Let $(u,\lambda)$ and $(u_h,\lambda_h)$ denote the unique solutions of~\eqref{eq:dglm} and~\eqref{eq:dglm:disc}, respectively. Then, for any $T\in\cT$ the inequality
  \begin{align*}
    \rho(T)^2 &\lesssim \|u-u_h\|_{\Omega_T}^2 + \varepsilon^2\|\pwnabla(u-u_h)\|_{\Omega_T}^2 
    \\
    &\qquad + \min_{v_h\in P^0(T)} \|u-v_h\|_{T}^2 + \min_{\btau_h\in P^0(T)^d}\varepsilon^2\|\pwnabla u-\btau_h\|_{T}^2
    + \|(1-\Pi_T^0)f\|_T^2
  \end{align*}
  holds.\qed
\end{corollary}

\subsection{Error estimator for the DHFEM}\label{sec:apost:dualdglm}
For given $T\in\cT$ fix some vertex $z_T\in \cV_T$ and let $\cE_{T,z_T}$ denote the set of all $d$ distinct edges that share vertex $z_T$. 
To each of these edges we associate the tangential vector $\tangential_E^T\in\R^d$, $E\in \cE_{T,z_T}$ with unit length. Note that $\set{\tangential_E^T}{E\in\cE_{T,z_T}}$ spans $\R^d$.
Furthermore, let $\eta_E^T = \prod_{z\in\cV_E} \eta_z^T$ denote the edge bubble function.

Throughout this section we consider for each $T\in\cT$, 
\begin{align*}
  \Sigma_{h,\varepsilon}(T) = \begin{cases}
    \RT^0(T) + \linhull\set{\eeta_{F,\varepsilon}^T}{F\in\cF_T} + \linhull\set{\eta_E^T\tangential^T_E}{E\in\cE_{T,z_T}} & \text{if } \varepsilon<h_T, \\
    \RT^0(T) + \linhull\set{\eeta_{F}^T}{F\in\cF_T} + \linhull\set{\eta_E^T\tangential^T_E}{E\in\cE_{T,z_T}} & \text{if } \varepsilon\geq h_T.
  \end{cases}
\end{align*}
Define the product space
\begin{align*}
  \Sigma_h=\Sigma_{h,\varepsilon} = \prod_{T\in\cT} \Sigma_{h,\varepsilon}(T).
\end{align*}
We also consider space $W_h$ for $p=0$, i.e., $W_h = \trnabla_\cT(P^1(\cT)\cap H_0^1(\Omega))$. 

The following result follows from the analysis given in~\cite[Section~4]{FuehrerHeuerFortin}.

\begin{theorem}\label{thm:fortin:dual:aug}
  There exists $\Pi_\Sigma\colon \Sigma\to \Sigma_{h}$ satisfying~\eqref{eq:def:fortin} with constant $C_\Sigma$ independent of $\varepsilon$ and $h$, and satisfying the additional property
  \begin{align*}
    \ip{\bq}{\btau-\Pi_\Sigma \btau}_\Omega = 0 \quad\forall \bq\in P^0(\cT)^d, \,\btau\in \Sigma.
  \end{align*}
  \qed
\end{theorem}

For the remainder of this section let $(\bsigma_h,w_h)\in \Sigma_{h}\times W_h$ denote the unique solution of~\eqref{eq:dualdglm:disc}.
Define for each $T\in\cT$ the (squared) indicator
\begin{align*}
  \xi(T)^2 &= \|(1-\Pi_T^0)(\varepsilon\div\bsigma_h-f)\|_T^2 + \|(1-\Pi_T^0)\bsigma_h\|_T^2 + \min\{\varepsilon,h\}\|\jump{\bsigma_h\cdot\normal}\|_{\partial T\setminus\partial\Omega}^2.
\end{align*}
Note that $\xi(T)$ does not involve solution component $w_h$. 

In the next results we analyze reliability and efficiency of the estimator 
\begin{align*}
  \xi^2 = \sum_{T\in\cT} \xi(T)^2.
\end{align*}

\begin{theorem}\label{thm:dualdglm:apost}
  Let $(\bsigma,w)\in \Sigma\times W$ and $(\bsigma_h,w_h)\in \Sigma_h\times W_h$ denote the solutions of~\eqref{eq:dualdglm} and~\eqref{eq:dualdglm:disc}, respectively. 
  Estimator $\xi$ is reliable, i.e., there exists a constant $C_\mathrm{rel}$ independent of $\varepsilon$, $h$, $(\bsigma_h,w_h)$ such that
  \begin{align*}
    \|\bsigma-\bsigma_h\|_{\Sigma,\varepsilon} + \|w-w_h\|_{W,\varepsilon} \leq C_\mathrm{rel} \xi.
  \end{align*}
\end{theorem}
\begin{proof}
  The first part of the proof is similar to the proof of Theorem~\ref{thm:dglm:apost} which yields the estimate
  \begin{align*}
    \|\bsigma-\bsigma_h\|_{\Sigma,\varepsilon} + \|w-w_h\|_{W,\varepsilon}
    &\lesssim \|(1-\Pi_h^0)(\varepsilon\pwdiv\bsigma_h-f)\|_\Omega^2 + \|(1-\Pi_h^0)\bsigma_h\|_\Omega^2
    \\
    &\qquad\qquad + \sup_{0\neq w\in W} \frac{\varepsilon\dual{(\bsigma-\bsigma_h)\cdot\normal}{w}_{\partial\cT}}{\|w\|_{W,\varepsilon}}.
  \end{align*}
  It only remains to localize the last term on the right-hand side. 
  Note that by continuity of normal components of $\bsigma$ we have that $\dual{\bsigma\cdot\normal}{w}_{\partial\cT} = 0$ for all $w\in W$.
  Given $w\in W$, let $\widetilde w\in H_0^1(\Omega)$ be such that $\trnabla_\cT \widetilde w=w$ and $\|\widetilde w\|_{1,\varepsilon,\Omega} = \|w\|_{W,\varepsilon}$.
  Further, let $w_h = \trnabla_\cT J_{\cT,0} \widetilde w \in W_h$.
  Then, \eqref{eq:dualdglm:disc:b} implies that $\dual{\bsigma_h\cdot\normal}{w_h}_{\partial\cT} = 0$. Therefore, 
  \begin{align*}
    \varepsilon\dual{(\bsigma-\bsigma_h)\cdot\normal}{w}_{\partial\cT} = -\varepsilon\dual{\bsigma_h\cdot\normal}{w-w_h}_{\partial\cT}
    &= \sum_{F\in\cF_\mathrm{int}} -\varepsilon\dual{\jump{\bsigma_h\cdot\normal}}{w-w_h}_F
    \\
    &\leq \sum_{F\in\cF_\mathrm{int}} \varepsilon\|\jump{\bsigma_h\cdot\normal}\|_F\|w-w_h\|_F.
  \end{align*}
  Here, $\cF_\mathrm{int}\subset \cF$ denotes the set of interior facets.
  With the properties of operator $J_{\cT,0}$ given in~\eqref{eq:propquasiint} we may estimate 
  \begin{align*}
    \varepsilon\|w-w_h\|_F \lesssim h_T^{1/2}\varepsilon\|\nabla \widetilde w\|_{\Omega_T} \quad \text{for } T \text{ with } F\in\cF_T.
  \end{align*}
  Alternatively, we may invoke the trace inequality in a different way and estimate the very same term by
  \begin{align*}
    \varepsilon\|w-w_h\|_F &\lesssim \frac{1}{h_T^{1/2}}\|(1-J_{\cT,0})\widetilde w\|_T^{1/2}(\|(1-J_{\cT,0})\widetilde w\|_T+h_T\|\nabla(1-J_{\cT,0})\widetilde w\|_T)^{1/2}
    \\
    &\lesssim 
    \varepsilon^{1/2}\|\widetilde w\|_{\Omega_T}^{1/2}\varepsilon^{1/2}\|\nabla\widetilde w\|_{\Omega_T}^{1/2}
    \lesssim \varepsilon^{1/2}(\|\widetilde w\|_{\Omega_T} + \varepsilon\|\nabla\widetilde w\|_{\Omega_T}).
  \end{align*}
  We conclude that
  \begin{align*}
    \varepsilon\|w-w_h\|_F \lesssim \min\{\varepsilon^{1/2},h_T^{1/2}\} \|\widetilde w\|_{1,\varepsilon,\Omega_T}.
  \end{align*}
  Consequently, 
  \begin{align*}
    \sup_{0\neq w\in W} \frac{\varepsilon\dual{(\bsigma-\bsigma_h)\cdot\normal}{w}_{\partial\cT}}{\|w\|_{W,\varepsilon}}
    \lesssim \left(\sum_{T\in\cT} \min\{\varepsilon,h_T\}\|\jump{\bsigma_h\cdot\normal}\|_{\partial T\setminus\partial\Omega}^2\right)^{1/2}.
  \end{align*}
  This finishes the proof.
\end{proof}

To analyze an efficiency-type estimate we need the following technical result. 
We follow similar steps as in~\cite{Verf94} but to tackle the parameter-dependent norms we make use of the modified face bubble functions defined above. 
\begin{lemma}\label{lem:effTraceSigma}
  Let $T\in\cT$. Then, 
  \begin{align*}
    \min\{\varepsilon^{1/2},h_T^{1/2}\}\|\jump{\bsigma_h\cdot\normal}\|_{\partial T} \lesssim \|\bsigma-\bsigma_h\|_{\Omega_T} + \varepsilon \|\div\bsigma-\pwdiv\bsigma_h\|_{\Omega_T}.
  \end{align*}
\end{lemma}
\begin{proof}
  First, observe that $\jump{\bsigma_h\cdot\normal}|_F \in P^2(F)$. We recall that there exists an extension operator $P\colon P^2(F)\to C^0(\Omega_F)$ such that
  \begin{align*}
    \diam(F)\|\sigma\|_F^2 &\eqsim \diam(F) \|\sigma\eta_F^{1/2}\|_F^2 \eqsim \|P\sigma\|_{\Omega_F}^2, \qquad
    \|\nabla P\sigma\|_{\Omega_F} \lesssim \diam(F)^{-1/2} \|\sigma\|_F, \\
    \|P\sigma\|_{L_\infty(\Omega_F)} &\lesssim |F|^{-1/2}\|\sigma\|_F.
  \end{align*}
  We consider two cases, $h_T\leq \varepsilon$ and $\varepsilon<h_T$. 
  First, suppose $h_T\leq \varepsilon$. Verf\"urth's bubble function technique yields
  \begin{align*}
    h_T^{1/2}\|\jump{\bsigma_h\cdot\normal}\|_{\partial T} \lesssim \|\bsigma-\bsigma_h\|_{\Omega_T} + h_T\|\div\bsigma-\pwdiv\bsigma\|_{\Omega_T} \leq \|\bsigma-\bsigma_h\|_{\Omega_T} + \varepsilon \|\div\bsigma-\pwdiv\bsigma\|_{\Omega_T}.
  \end{align*}
  Second, suppose $\varepsilon < h_T$. Let $F\in \cF_T$ be given. 
  Set $\sigma = \varepsilon^{1/2}\jump{\bsigma\cdot\normal}|_F\in P^2(F)$. Then, 
  \begin{align*}
    \varepsilon \|\jump{\bsigma\cdot\normal}\|_F^2 &= \varepsilon^{1/2}\dual{\jump{\bsigma\cdot\normal}}{\sigma}_F
    \eqsim \varepsilon^{1/2}\dual{\jump{\bsigma\cdot\normal}}{\sigma\eta_{F,\varepsilon}}_F
    \\
    &= \varepsilon^{1/2}\ip{\bsigma-\bsigma_h}{\nabla(\eta_{F,\varepsilon}P\sigma)}_{\Omega_F} + 
    \varepsilon^{1/2}\ip{\pwdiv(\bsigma-\bsigma_h)}{\eta_{F,\varepsilon}P\sigma}_{\Omega_F}.
  \end{align*}
  Here we set $\eta_{F,\varepsilon}|_{T'}= \eta_{F,\varepsilon}^{T'}$ for all $T'\in\cT$ with $F\in\cF_{T'}$.
  Using the scaling properties (Lemma~\ref{lem:propEtaFe}) of the modified face bubble functions and the properties of the extension operator $P$ given at the beginning of the proof, we estimate the last terms on the right-hand side as follows:
  \begin{align*}
    \varepsilon^{1/2}\ip{\pwdiv(\bsigma-\bsigma_h)}{\eta_{F,\varepsilon}P\sigma}_{\Omega_F}
    &\lesssim \varepsilon^{1/2}\|\pwdiv(\bsigma-\bsigma_h)\|_{\Omega_F}\|P\sigma\|_{L_\infty(\Omega_F)}\frac{\varepsilon^{1/2}}{h_T^{1/2}} |\Omega_F|^{1/2}\\
    &\lesssim \varepsilon \|\pwdiv(\bsigma-\bsigma_h)\|_{\Omega_F} \|\sigma\|_F, \\
    \varepsilon^{1/2}\ip{\bsigma-\bsigma_h}{\nabla(\eta_{F,\varepsilon}P\sigma)}_{\Omega_F} &=
    \varepsilon^{1/2}\ip{\bsigma-\bsigma_h}{\nabla(\eta_{F,\varepsilon})P\sigma}_{\Omega_F}
    + \varepsilon^{1/2}\ip{\bsigma-\bsigma_h}{\eta_{F,\varepsilon}\nabla P\sigma}_{\Omega_F}
    \\
    &\lesssim \|\bsigma-\bsigma_h\|_{\Omega_F}\varepsilon^{1/2}\frac{h_T^{1/2}}{\varepsilon^{1/2}} |T|^{1/2} h_T^{-1}\|P\sigma\|_{L_\infty(\Omega_F)} 
    \\
    &\qquad+ \|\bsigma-\bsigma_h\|_{\Omega_F}\varepsilon^{1/2}\|\nabla P\sigma\|_{\Omega_F}
    \\
    &\lesssim \|\bsigma-\bsigma_h\|_{\Omega_F}\|\sigma\|_F.
  \end{align*}
  In the last estimate we also used that $\varepsilon<h_T$. Putting all estimates together concludes the proof.
\end{proof}

The next results follows directly from the last one and the triangle inequality.
\begin{corollary}
  Let $(\bsigma,w)$ and $(\bsigma_h,w_h)$ denote the unique solutions of~\eqref{eq:dualdglm} and~\eqref{eq:dualdglm:disc}, respectively. Then, for any $T\in\cT$ the inequality
  \begin{align*}
    \xi(T)^2 &\lesssim \|\bsigma-\bsigma_h\|_{\Omega_T}^2 + \varepsilon^2\|\pwdiv(\bsigma-\bsigma_h)\|_{\Omega_T}^2 
    \\
    &\qquad + \min_{\btau_h\in P^0(T)^d} \|\bsigma-\btau_h\|_{T}^2 + \min_{v_h\in P^0(T)}\varepsilon^2\|\pwdiv \bsigma-v_h\|_{T}^2
    + \|(1-\Pi_T^0)f\|_T^2
  \end{align*}
  holds.\qed
\end{corollary}

\section{Numerical experiments}\label{sec:numeric}
In this section we present numerical experiments using the two hybrid methods introduced above in Sections~\ref{sec:dglm} and~\ref{sec:dualdglm}, respectively. 
We compare our results to the lowest-order continuous Galerkin method given by: Find $\uhcg\in P^1(\cT)\cap H_0^1(\Omega)$ such that
\begin{align}\label{eq:fem}
  \varepsilon^2\ip{\nabla\uhcg}{\nabla v}_\Omega + \ip{\uhcg}v_\Omega = \ip{f}{v}_\Omega \quad\forall v\in P^1(\cT)\cap H_0^1(\Omega).
\end{align}
Throughout, we solve the PHFEM with spaces $U_h\times\Lambda_h$ as given in Section~\ref{sec:apost:dglm}, and the DHFEM with spaces $\Sigma_h\times W_h$ as given in Section~\ref{sec:apost:dualdglm}.

\subsection{Manufactured solution on square domain}\label{sec:numeric:ex1}
Let $\Omega = (0,1)^2$. We consider the manufactured solution (see also~\cite[Section~5.2]{FuehrerHeuerFortin})
\begin{align*}
  u(x,y) = v(x)v(y), 
  \quad
  v(t) = 1-(1-e^{-1/(\sqrt{2}\varepsilon)})\frac{e^{-(1-t)/(\sqrt{2}\varepsilon)}+e^{-t/(\sqrt{2}\varepsilon)}}{1-e^{-2/(\sqrt{2}\varepsilon)}}.
\end{align*}
Function $u$ satisfies model problem~\eqref{eq:model} with
\begin{align*}
  f(x,y) = \frac12(v(x)+v(y)). 
\end{align*}

In Figure~\ref{fig:exp1:dglmFEM} we compare $\Pi_h^0 u_h$ where $(u_h,\lambda_h)\in U_h\times \Lambda_h$ is the solution of~\eqref{eq:dglm:disc}, to the solution $\uhcg$ of~\eqref{eq:fem}. Figure~\ref{fig:exp1:dualdglmFEM} visualizes $\Pi_h^0u_h^\mathrm{dual}$ on two meshes where we recall that $u_h^\mathrm{dual} = \varepsilon\pwdiv\bsigma_h+f$ with $(\bsigma_h,w_h)\in \Sigma_h\times W_h$ the solution of~\eqref{eq:dualdglm:disc}. In all cases we have used $\varepsilon=10^{-4}$.
We find that both $\Pi_h^0 u_h$ and $\Pi_h^0 u_h^\mathrm{dual}$ only show marginal oscillations whereas $\uhcg$ shows considerable oscillations close to the boundary of the domain.

\begin{figure}
  \begin{center}
    \input{Example1sol}
  \end{center}
  \caption{Comparison of the projected PHFEM solution component $\Pi_h^0u_h$ and the lowest-order continuous Galerkin solution $\uhcg$ on meshes with $\#\cT = 64$ (upper row) and $\#\cT=1024$ (lower row).}
  \label{fig:exp1:dglmFEM}
\end{figure}
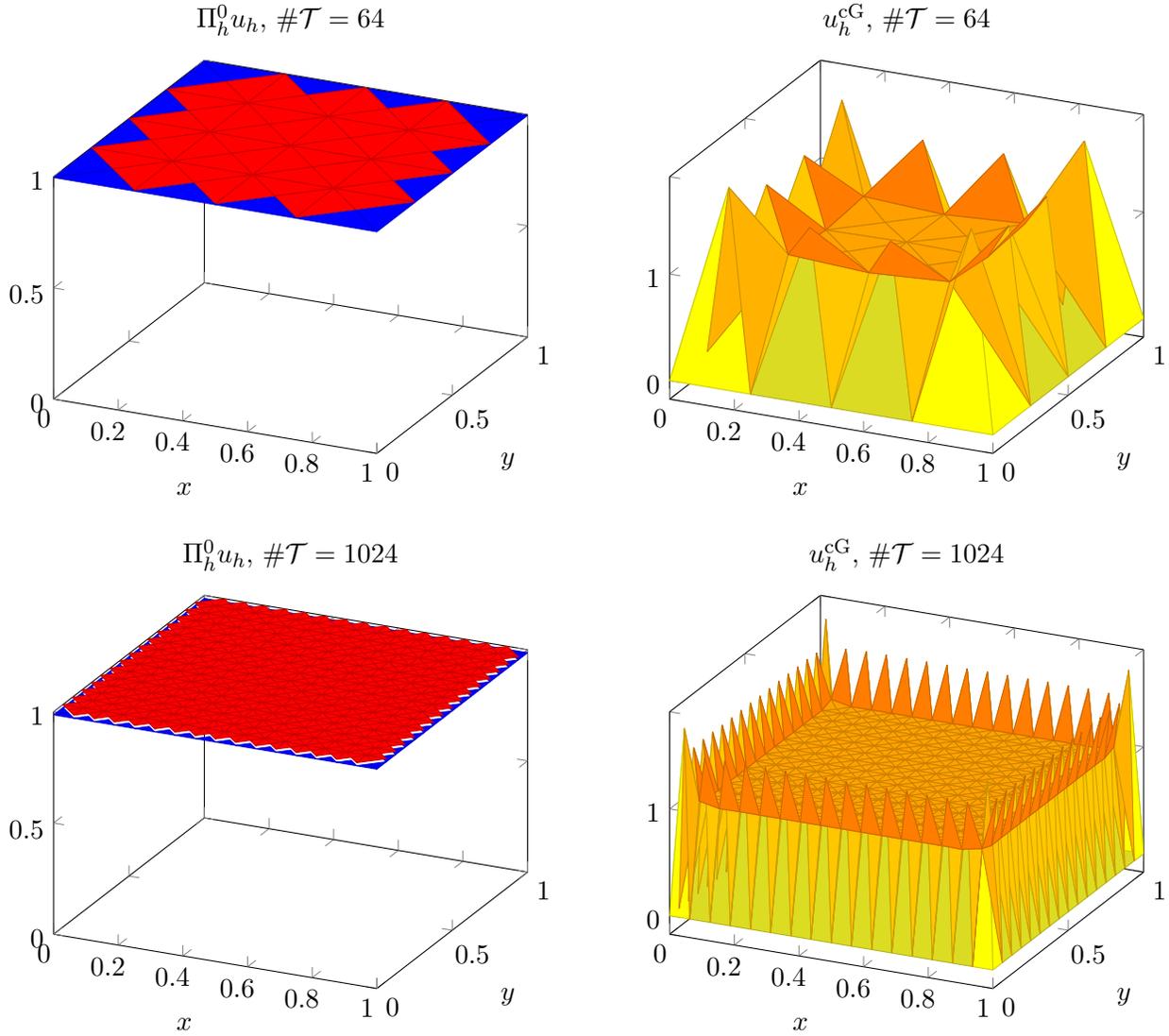

\begin{figure}
  \begin{center}
    \input{Example1DualSol}
  \end{center}
  \caption{Postprocessed solution $\Pi_h^0u_h^\mathrm{dual}$ of the DHFEM on two meshes and $\varepsilon=10^{-4}$.}
  \label{fig:exp1:dualdglmFEM}
\end{figure}
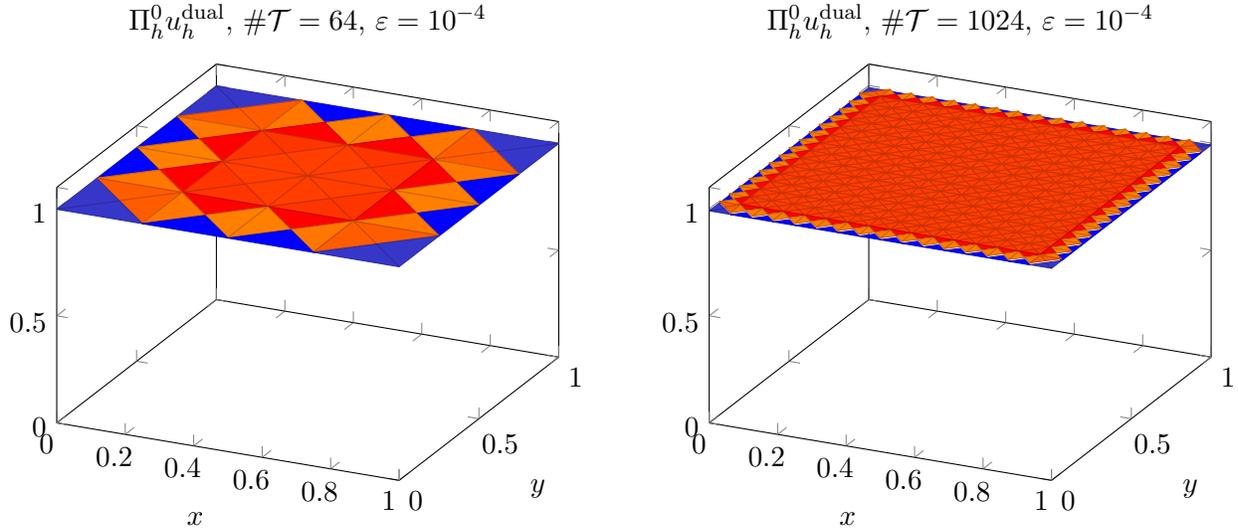

This is also reflected in the $L_2(\Omega)$ errors as depicted in Figure~\ref{fig:exp1:errL2}. For coarse meshes there is a significant difference between the continuous Galerkin method and the PHFEM resp. DHFEM. 
Note that $\#\mathrm{dof} = \dim(\Lambda_h) = \#\cF$ for the PHFEM, whereas $\#\mathrm{dof} = \dim(W_h) = \dim(P^1(\cT)\cap H_0^1(\Omega)) = \#\cV_0$ (number of interior vertices of the mesh $\cT$).
Note that the left plot shows the $L_2(\Omega)$ errors of projections onto $P^0(\cT)$, whereas the right plot shows the errors of the projections onto $P^1(\cT)$.

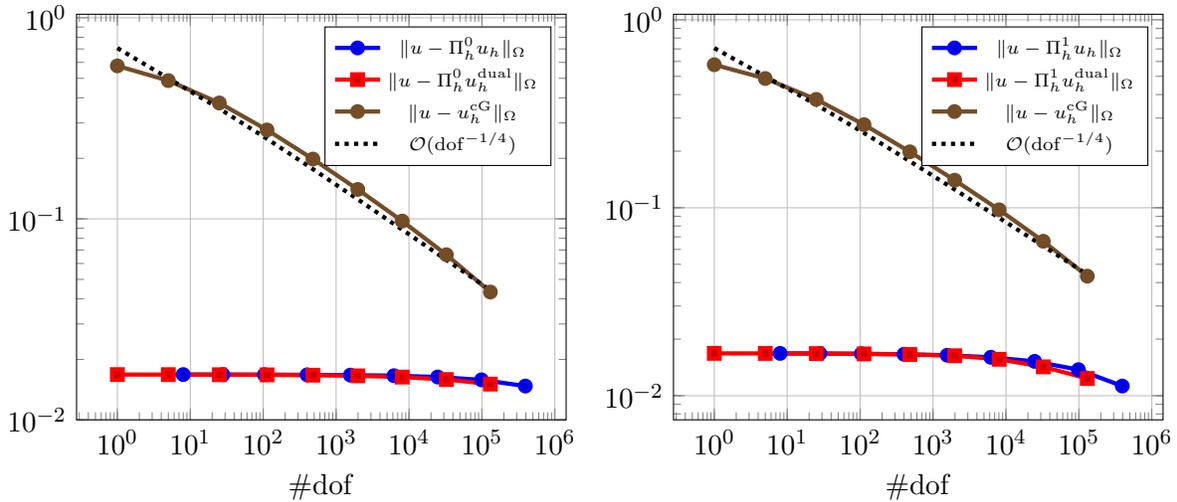
\begin{figure}
  \begin{center}
    \input{Example1ErrL2}
  \end{center}
  \caption{Comparison between $L_2$ errors of the PHFEM, DHFEM solution and continuous Galerkin method for the example from Section~\ref{sec:numeric:ex1} with $\varepsilon=10^{-4}$.}
  \label{fig:exp1:errL2}
\end{figure}

\subsection{Example with unknown solution}
Let $\Omega = (-1,1)^2$ and define for $(x,y)\in \Omega)$, 
\begin{align*}
  f(x,y)  = \begin{cases}
    1 & \text{if } (x,y)\in (-1/2,1/2)^2, \\
    -1 &\text{else}.
  \end{cases}
\end{align*}
For this right-hand side datum we do not have an explicit representation of the solution $u$ to~\eqref{eq:model}. 
We use a simple adaptive loop which consists of the steps \textbf{Solve}, \textbf{Estimate}, \textbf{Mark}, \textbf{Refine}. 
To mark elements for refinement we use the bulk criterion: Find a (minimal) set $\mathcal{M}\subset \cT$ such that
\begin{align}\label{eq:bulk}
  \theta \est^2 \leq \sum_{T\in\mathcal{M}} \est(T)^2.
\end{align}
Here, $\est$ stands for an error estimator with local contributions $\est(T)$. For the PHFEM~\eqref{eq:dglm:disc} we use $\est=\rho$.
For this problem we expect the exact solution to have boundary layers and interior layers (at $\partial (-1/2,1/2)^2$). 
The adaptive algorithm seems to detect these layers and generates meshes that are highly refined in those regions.
We stress that we use the newest-vertex bisection algorithm for mesh-refinement which generates shape-regular triangulations.
Figure~\ref{fig:expUnknown:dglmFEM} shows $\Pi_h^0 u_h$ and $\uhcg$ on a locally refined mesh generated by the adaptive algorithm and with $\varepsilon=10^{-8}$. 
We find that $\Pi_h^0u_h$ does not show any visible oscillations, whereas $\uhcg$ does.
We note that similar observations are true for $\Pi_h^0u_h^\mathrm{dual}$ and using estimator $\est = \xi$. To keep a shorter presentation we do not report the results here. 

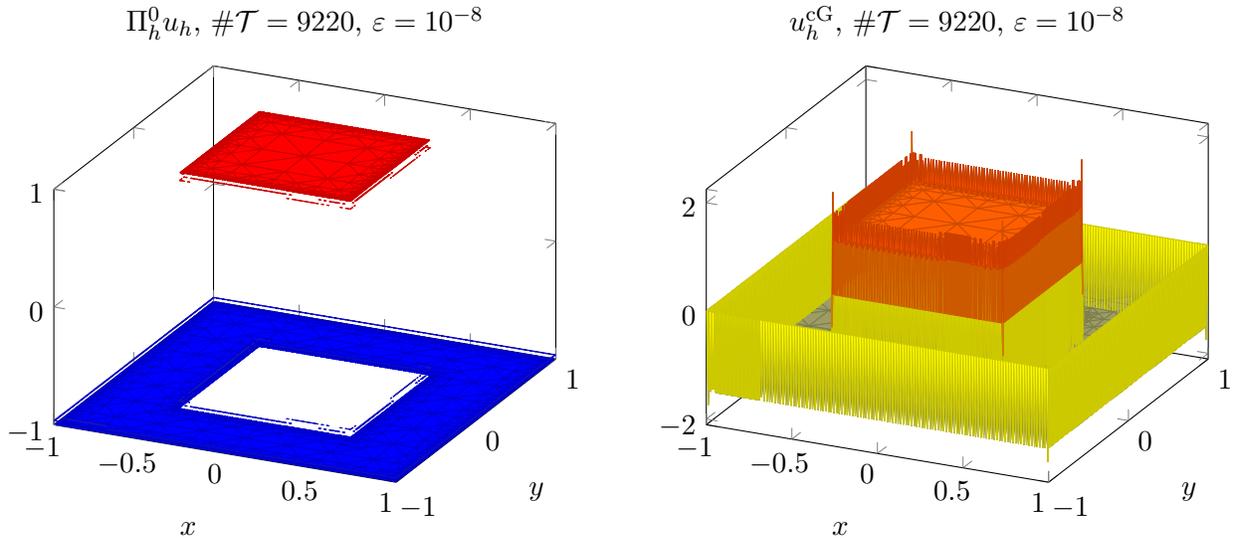
\begin{figure}
  \begin{center}
    \input{ExampleUnknown}
  \end{center}
  \caption{Comparison of the projected PHFEM solution component $\Pi_h^0u_h$ and the lowest-order continuous Galerkin solution $\uhcg$ on a locally refined mesh.}
  \label{fig:expUnknown:dglmFEM}
\end{figure}

Finally, Figure~\ref{fig:expUnknown:est} shows the error estimators for both proposed methods for $\varepsilon = 10^{-3}$, $\varepsilon = 10^{-4}$ on a sequence of uniformly refined meshes as well as adaptively refined meshes. Again, adaptive mesh-refinement is steered by estimator $\est = \rho$ resp. $\est=\xi$ and parameter $\theta =0.25$ in the bulk criterion~\eqref{eq:bulk}.
The use of adaptive mesh-refinements significantly reduces the pre-asymptotic convergence range. 
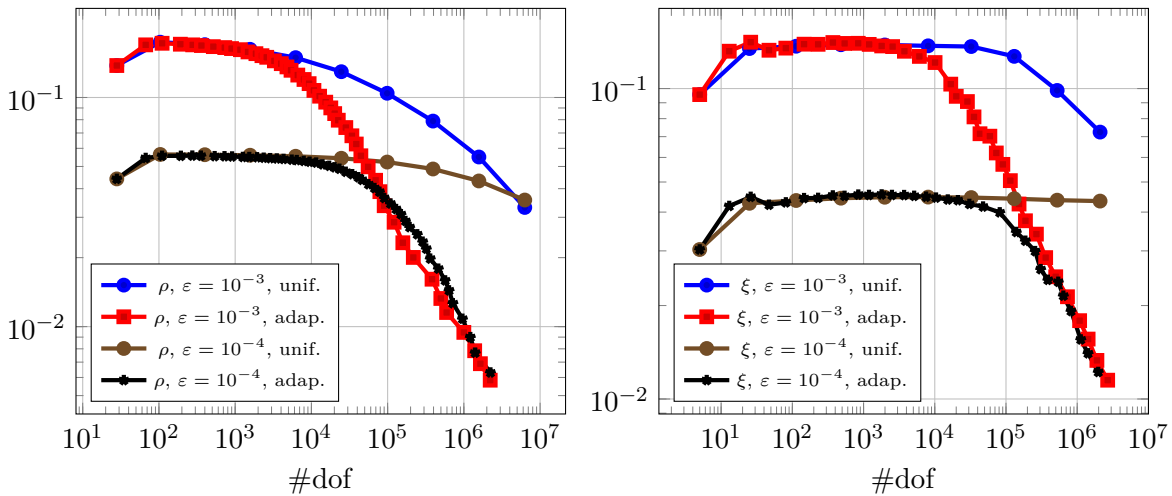
\begin{figure}
  \begin{center}
    \input{ExampleUnknownEst}
  \end{center}
  \caption{Estimators for the PHFEM (left) and DHFEM (right) on a sequence of uniformly (unif.) and adaptively (adap.) refined meshes.}
  \label{fig:expUnknown:est}
\end{figure}

\section{Conclusions}
Two hybrid finite element methods have been presented and analyzed. 
The first one is based on a primal formulation and the second one is based on a dual formulation.
To achieve uniform stability on shape-regular simplicial meshes, i.e. $\inf$--$\sup$ constants independent of the mesh-size and the singular perturbation parameter $\varepsilon$, we enriched the local spaces by modified face bubble functions. These function are equal to the standard polynomial face bubble functions on element boundaries but decay exponentially in the interior of elements. An alternative for defining modified face bubble functions involving only polynomial functions has been presented as well. 

For both methods the algebraic system to solve is of the dimension of the (discrete) Lagrange multiplier space, hence, comparable to common finite element methods.
Numerical experiments show no significant spurious oscillations which is also reflected in much smaller errors compared to the continuous Galerkin method. 

A posteriori error estimators are derived that control the error in the naturally induced norms of the problem independent of the singular perturbation parameter $\varepsilon$. 
Numerical experiments indicate that adaptive mesh-adaptation based on these estimators detects boundary and interior layers and is more efficient compared to uniform mesh-refinements. 

For future works we plan to study non-polynomial discretizations for the Lagrange multiplier spaces.
Furthermore, we like to investigate if some techniques developed in this work can be applied to other singularly perturbed problems, e.g., advection-dominated diffusion problems, as well as to more 
general multiscale problems,  for instance by setting highly oscillatory data parameters.

\bibliographystyle{alpha}
\bibliography{literature}
\end{document}

%% file: header.tex
\newtheorem{theorem}{Theorem}

\newtheorem{lemma}[theorem]{Lemma}
\newtheorem{corollary}[theorem]{Corollary}
\newtheorem{proposition}[theorem]{Proposition}

\newcommand{\uhcg}{u_h^\mathrm{cG}}

\DeclareMathOperator{\linhull}{span}

\newcommand{\ip}[2]{(#1\hspace*{.5mm},#2)}
\newcommand{\dual}[2]{\langle#1\hspace*{.5mm},#2\rangle}

\newcommand{\diam}{\mathrm{diam}}

\newcommand\curl{\operatorname{curl}}

\def\div{{\rm div\,}}

\newcommand{\jump}[1]{[#1]}

\newcommand{\Hdivset}[1]{\boldsymbol{H}(\div;#1)}

\newcommand{\set}[2]{\big\{#1\,:\,#2\big\}}

\newcommand{\pwnabla}{\operatorname{\nabla_{\cT}}}
\newcommand{\pwdiv}{\operatorname{div_{\cT}}}

\newcommand{\RT}{\ensuremath{\boldsymbol{RT}}}

\newcommand{\R}{\ensuremath{\mathbb{R}}}
\newcommand{\N}{\ensuremath{\mathbb{N}}}

\newcommand{\cT}{\ensuremath{\mathcal{T}}}
\newcommand{\cF}{\ensuremath{\mathcal{F}}}
\newcommand{\cE}{\ensuremath{\mathcal{E}}}

\newcommand{\OO}{\ensuremath{\mathcal{O}}}

\newcommand{\normal}{\ensuremath{{\boldsymbol{n}}}}
\newcommand{\tangential}{\ensuremath{{\boldsymbol{t}}}}

\newcommand{\cV}{\ensuremath{\mathcal{V}}}

\newcommand{\bsigma}{{\boldsymbol\sigma}}
\newcommand{\eeta}{{\boldsymbol\eta}}
\newcommand{\btau}{{\boldsymbol\tau}}

\newcommand{\bq}{{\boldsymbol{q}}}
\newcommand{\bx}{{\boldsymbol{x}}}

\newcommand{\tr}{\operatorname{tr}}

\newcommand{\trdiv}{\tr^{\div}}
\newcommand{\trnabla}{\tr^{\nabla}}

\newcommand{\est}{\operatorname{est}}

%% file: Example1sol.tex
\begin{tikzpicture}
  \begin{groupplot}[
      group style={group size=2 by 2, horizontal sep=2cm, vertical sep=2cm},
      width=0.5\textwidth,
      ylabel={$y$},
      xlabel={$x$},
    ]
    \nextgroupplot[title={$\Pi_h^0u_h$, $\#\cT=64$},zmin=0,zmax=1]
      \addplot3[patch] table{data/Example1DGLM_solP0_00064.dat};
    \nextgroupplot[title={$\uhcg$, $\#\cT=64$}]
      \addplot3[patch] table{data/Example1DGLM_solS1_00064.dat};
    \nextgroupplot[title={$\Pi_h^0u_h$, $\#\cT=1024$},zmin=0,zmax=1]
      \addplot3[patch] table{data/Example1DGLM_solP0_01024.dat};
    \nextgroupplot[title={$\uhcg$, $\#\cT=1024$}]
      \addplot3[patch] table{data/Example1DGLM_solS1_01024.dat};
  \end{groupplot}
\end{tikzpicture}

%% file: Example1DualSol.tex
\begin{tikzpicture}
  \begin{groupplot}[
      group style={group size=2 by 1, horizontal sep=2cm, vertical sep=2cm},
      width=0.5\textwidth,
      ylabel={$y$},
      xlabel={$x$},
      zmin=0,
      zmax=1.1,
    ]
    \nextgroupplot[title={$\Pi_h^0u_h^\mathrm{dual}$, $\#\cT=64$, $\varepsilon=10^{-4}$}]
      \addplot3[patch] table{data/Example1DualDGLM_solP0_00064.dat};
      \nextgroupplot[title={$\Pi_h^0u_h^\mathrm{dual}$, $\#\cT=1024$, $\varepsilon=10^{-4}$}]
      \addplot3[patch] table{data/Example1DualDGLM_solP0_01024.dat};
  \end{groupplot}
\end{tikzpicture}

%% file: Example1ErrL2.tex
\begin{tikzpicture}
\begin{loglogaxis}[
width=0.49\textwidth,
xlabel={$\#\mathrm{dof}$},
grid=major,
legend entries={{\tiny $\|u-\Pi_h^0u_h\|_\Omega$},{\tiny $\|u-\Pi_h^0u_h^\mathrm{dual}\|_\Omega$},{\tiny $\|u-\uhcg\|_\Omega$},{\tiny $\OO(\mathrm{dof}^{-1/4})$}},
legend pos=north east,
every axis plot/.append style={ultra thick},
]
\addplot table [x=dof,y=errUP0L2] {data/Example1_DGLM_errors.dat};
\addplot table [x=dof,y=errUP0L2] {data/Example1_DualDGLM_errors.dat};
\addplot table [x=dof,y=errUS1L2] {data/Example1_DualDGLM_errors.dat};
\addplot [black,dotted,mark=none] table [x=dof,y expr={sqrt(\thisrowno{1})^(-1/2)}] {data/Example1_DualDGLM_errors.dat};

\end{loglogaxis}
\end{tikzpicture}
\begin{tikzpicture}
\begin{loglogaxis}[
width=0.49\textwidth,
xlabel={$\#\mathrm{dof}$},
grid=major,
legend entries={{\tiny $\|u-\Pi_h^1u_h\|_\Omega$},{\tiny $\|u-\Pi_h^1u_h^\mathrm{dual}\|_\Omega$},{\tiny $\|u-\uhcg\|_\Omega$},{\tiny $\OO(\mathrm{dof}^{-1/4})$}},
legend pos=north east,
every axis plot/.append style={ultra thick},
]
\addplot table [x=dof,y=errUP1L2] {data/Example1_DGLM_errors.dat};
\addplot table [x=dof,y=errUP1L2] {data/Example1_DualDGLM_errors.dat};
\addplot table [x=dof,y=errUS1L2] {data/Example1_DualDGLM_errors.dat};
\addplot [black,dotted,mark=none] table [x=dof,y expr={sqrt(\thisrowno{1})^(-1/2)}] {data/Example1_DualDGLM_errors.dat};

\end{loglogaxis}
\end{tikzpicture}

%% file: ExampleUnknown.tex
\begin{tikzpicture}
  \begin{groupplot}[
      group style={group size=2 by 1, horizontal sep=2cm, vertical sep=2cm},
      width=0.5\textwidth,
      ylabel={$y$},
      xlabel={$x$},
    ]
    \nextgroupplot[title={$\Pi_h^0u_h$, $\#\cT=9220$, $\varepsilon=10^{-8}$},zmin=-1,zmax=1]
      \addplot3[patch] table{data/ExampleUnknown_DGLM_solP0_t1e-4_09153.dat};
      \nextgroupplot[title={$\uhcg$, $\#\cT=9220$, $\varepsilon=10^{-8}$}]
      \addplot3[patch] table{data/ExampleUnknown_DGLM_solS1_t1e-4_09153.dat};
  \end{groupplot}
\end{tikzpicture}

%% file: ExampleUnknownEst.tex
\begin{tikzpicture}
\begin{loglogaxis}[
width=0.49\textwidth,
xlabel={$\#\mathrm{dof}$},
grid=major,
legend entries={{\tiny $\rho$, $\varepsilon=10^{-3}$, unif.},{\tiny $\rho$, $\varepsilon=10^{-3}$, adap.},{\tiny $\rho$, $\varepsilon=10^{-4}$, unif.},{\tiny $\rho$, $\varepsilon=10^{-4}$, adap.}},
legend pos=south west,
every axis plot/.append style={ultra thick},
]
\addplot table [x=dof,y=est] {data/ExampleUnknown_DGLM_est3unif.dat};
\addplot table [x=dof,y=est] {data/ExampleUnknown_DGLM_est3adap.dat};
\addplot table [x=dof,y=est] {data/ExampleUnknown_DGLM_est4unif.dat};
\addplot table [x=dof,y=est] {data/ExampleUnknown_DGLM_est4adap.dat};

\end{loglogaxis}
\end{tikzpicture}
\begin{tikzpicture}
\begin{loglogaxis}[
width=0.49\textwidth,
xlabel={$\#\mathrm{dof}$},
grid=major,
legend entries={{\tiny $\xi$, $\varepsilon=10^{-3}$, unif.},{\tiny $\xi$, $\varepsilon=10^{-3}$, adap.},{\tiny $\xi$, $\varepsilon=10^{-4}$, unif.},{\tiny $\xi$, $\varepsilon=10^{-4}$, adap.}},
legend pos=south west,
every axis plot/.append style={ultra thick},
]
\addplot table [x=dof,y=est] {data/ExampleUnknown_DualDGLM_est3unif.dat};
\addplot table [x=dof,y=est] {data/ExampleUnknown_DualDGLM_est3adap.dat};
\addplot table [x=dof,y=est] {data/ExampleUnknown_DualDGLM_est4unif.dat};
\addplot table [x=dof,y=est] {data/ExampleUnknown_DualDGLM_est4adap.dat};

\end{loglogaxis}
\end{tikzpicture}